\documentclass[12pt,reqno,a4paper,twoside]{article}

\usepackage{amsmath,amsthm,amstext,amscd,amssymb,url}

\usepackage{vmargin}
\setpapersize{A4} \setmarginsrb{25mm}{20mm}{20mm}{15mm}
{0mm}{10mm}{0mm}{10mm}

\usepackage{amssymb} 
\usepackage{amsmath}
\usepackage{textcomp}
\usepackage{lmodern}
\usepackage{xfrac}
\usepackage[latin1]{inputenc}

\usepackage{mathrsfs}
\usepackage{epsfig}
\usepackage[inline]{enumitem}

\newcommand{\vi}{\varphi}

\newcommand{\R}{\mathbb R}
\newcommand{\N}{\mathbb N}

\newcommand{\E}{\mathbb E}
\newcommand{\Zd}{\mathbb Z^d}

\newcommand{\pee}{\ensuremath{\mathbb{P}}}

\def\1{{\mathchoice {\rm 1\mskip-4mu l} {\rm 1\mskip-4mu l}
{\rm 1\mskip-4.5mu l} {\rm 1\mskip-5mu l}}}

\newtheorem{theorem}{{\small T}{\scriptsize HEOREM}}[section]
\newtheorem{corollary}{{\bf{\small C}{\scriptsize OROLLARY}}}[section]
\newtheorem{proposition}{{\bf{\small P}{\scriptsize ROPOSITION}}}[section]
\newtheorem{lemma}{{\bf{\small L}{\scriptsize EMMA}}}[section]
\newtheorem{remark}{{\bf{\small R}{\scriptsize EMARK}}}[section]
\newtheorem{definition}{{\bf{\small D}{\scriptsize EFINITION}}}[section]

\renewenvironment{proof}[1]
{\noindent{{\bf{\small{ P}{\scriptsize ROOF}}}.}\hspace{0.1cm} #1} {$\;\qed$\newline}

\newcommand{\beq}{\begin{eqnarray}}
\newcommand{\eeq}{\end{eqnarray}}

\newcommand{\ba}{\begin{align*}}
\newcommand{\ea}{\end{align*}}

\newcommand{\be}{\begin{equation}}
\newcommand{\ee}{\end{equation}}

\newcommand{\bl}{\begin{lemma}}
\newcommand{\el}{\end{lemma}}

\newcommand{\br}{\begin{remark}}
\newcommand{\er}{\end{remark}}

\newcommand{\bt}{\begin{theorem}}
\newcommand{\et}{\end{theorem}}

\newcommand{\bd}{\begin{definition}}
\newcommand{\ed}{\end{definition}}

\newcommand{\bp}{\begin{proposition}}
\newcommand{\ep}{\end{proposition}}

\newcommand{\bc}{\begin{corollary}}
\newcommand{\ec}{\end{corollary}}

\newcommand{\bpr}{\begin{proof}}
\newcommand{\epr}{\end{proof}}

\newcommand{\bi}{\begin{itemize}}
\newcommand{\ei}{\end{itemize}}

\newcommand{\ben}{\begin{enumerate}}
\newcommand{\een}{\end{enumerate}}

\newcommand{\caD}{{\mathcal D}}

\newcommand{\caF}{{\mathcal F}}
\newcommand{\caG}{{\mathcal G}}

\newcommand{\caL}{{\mathscr L}}
\newcommand{\caM}{{\mathcal M}}

\newcommand{\caS}{{\mathscr S}}

\newcommand{\BEP}{\text{\normalfont BEP}}
\newcommand{\IRW}{\text{\normalfont IRW}}
\newcommand{\SIP}{\text{\normalfont SIP}}

\newcommand{\SEP}{\text{\normalfont SEP}}

\newcommand{\eft}{\text{\tiny left}}
\newcommand{\ight}{\text{\tiny right}}

\newcommand{\seppar}{\gamma}

\newcommand{\HT}{\text{\normalfont [HT]}}

\newcommand{\Db}{\text{\normalfont [MD]}}
\newcommand{\BHT}{\text{\normalfont[BHT]}}
\newcommand{\UD}{\text{\normalfont[UD]}}
\newcommand{\const}{c}

\newmuskip\pFqmuskip

\newcommand*\pFq[6][8]{%
  \begingroup 
  \pFqmuskip=#1mu\relax
  \mathcode`\,=\string"8000
  \begingroup\lccode`\~=`\,
  \lowercase{\endgroup\let~}\pFqcomma
  {}_{#2}F_{#3}{\left[\genfrac..{0pt}{}{#4}{#5};#6\right]}%
  \endgroup
}
\newcommand{\pFqcomma}{\mskip\pFqmuskip}

\begin{document}
\title
{{\bf Factorized duality,
		stationary product measures and generating functions
}}
\author{Frank Redig
and Federico Sau
\\
\small{Delft Institute of Applied Mathematics}\\
\small{Delft University of Technology}\\
{\small Mekelweg 4, 2628 CD Delft}
\\
\small{The Netherlands}
}
\maketitle
\begin{abstract}

We  find all factorized duality functions for a class of interacting particle systems. The functions we recover are self-duality functions for interacting particle systems such as zero-range processes, symmetric inclusion and exclusion processes, as well as duality and self-duality functions for their continuous counterparts.
The approach is based on, firstly, a general relation between factorized duality functions and stationary product measures and, secondly,  an intertwining relation provided by generating functions.

 For the interacting particle systems,  these self-duality and duality functions turn out to be generalizations of those previously obtained in \cite{gkrv} and, more recently, in \cite{chiara}. Thus, we discover that only these two families of  dualities cover all possible cases. Moreover, the same method discloses all self-duality functions  for interacting diffusion systems such as the Brownian energy process, where both the process and its dual are in continuous variables.
\end{abstract}
\section{Introduction}
Duality and self-duality are very useful and powerful tools that allow to analyze properties of a complicated system in terms
of a simpler one. In case of self-duality for  particle systems, the dual system is the same and the simplification arises because in the dual
one considers only a finite number of  particles (e.g.\ \cite{demasi}).\\
Several methods are available to construct dual processes and duality relations. In the context of population dynamics,
the starting point is always to consider as dual the backward coalescent process (for a general overview, see \cite{dawson}). In the context of particle systems, the algebraic method
developed in \cite{gkrv} offers a framework to construct duality functions via symmetries of the generator and reversible measures.\\
However, a complete picture of how to obtain \emph{all} duality relations is missing.
In this latter context of interacting particle systems, natural questions of the same sort are:
which particle systems allow self-duality and is it possible to obtain  all factorized duality functions for these systems?\\
One of the useful applications of disposing of all duality functions is that, depending on the target, one can choose
appropriate ones: e.g.\ in the hydrodynamic limit and the study of the structure of the stationary measures, the ``classical'' duality functions are the appropriate ones (see e.g.\ \cite{demasi}), whereas in the study of (stationary and non-stationary) fluctuation fields and associated Boltzmann-Gibbs principles (\cite{landim}, chapter 11), as well as in the study of speed of relaxation to equilibrium in $L^2$ or in the study of perturbation theory around models with duality (cf.\ \cite{demasi}),  ``orthogonal'' duality functions turn out to be very useful.

In this paper, we develop an approach to answer the above questions and systematically determine all duality functions and relations for some interacting particle and diffusion systems. In this route,  starting from examples, we first investigate a general connection between stationary product measures and factorized duality functions. This shows, in particular, that for infinite systems  with factorized self-duality functions, the only stationary measures which are ergodic (w.r.t.\ either space-translation or time) are in fact product measures. \\
Then we use this connection between stationary product measures and duality functions to recover all possible factorized duality functions from the stationary product measures. More precisely, we show that, given the first duality function, i.e.\ the duality function with a single dual particle, all other duality functions are determined. This provides a simple machinery to obtain all duality functions in processes such as Symmetric Exclusion Process ($\SEP$), Symmetric Inclusion Process ($\SIP$) and
Independent Random Walkers ($\IRW$). In particular, we recover via this method all orthogonal polynomial duality functions obtained in \cite{chiara}.\\
Moreover, we prove that in the context of conservative particle systems  where the rates for particle hopping  depend only on the number of particles of the departure and arrival sites, the processes $\SEP$, $\SIP$ and $\IRW$ are the only systems which have self-duality with factorized self-duality functions and that
the first duality function is necessarily an affine function of the number of particles.\\
Next, in order to prove that  the ``possible'' duality functions derived via the method described above  are indeed duality functions, we develop a  method based on generating functions. This method, via an intertwining relation, allows to go from discrete systems (particle hopping dynamics) to continuous systems  (such as diffusion processes or deterministic dynamics) and back, and also allows to pass from self-duality to duality and back. The proof of a self-duality relation in a discrete system then reduces to the same property in a continuous system, which is much easier to check directly.\\
The generating function method also provides new examples of self-duality for processes in the continuum such as the Brownian Energy Process ($\BEP$), which intertwines with the $\SIP$ via the generating function. In fact, we show equivalence between self-duality of $\SIP$, duality between $\SIP$ and $\BEP$ and self-duality of $\BEP$.
Finally, this method based on generating functions generalizes the concept of obtaining dualities from symmetries to intertwinings, being a symmetry an intertwining of the generator with itself.

The paper is organized as follows. In Section \ref{section setting} we introduce the basic definitions of duality and systems considered. Additionally, in Theorem \ref{gzrprop} we prove which particle systems out of those considered admit factorized self-duality.\\
In Section \ref{section general relation}, we investigate  a general relation between factorized duality functions and stationary product measures. We treat separately the finite and infinite contexts in which this relation arises; in the latter case, we exploit this connection to draw some conclusions on the product structure of ergodic measures. \\
Section \ref{section general construction} is devoted to the derivation of all possible factorized self-duality and duality functions. Here Theorem \ref{gzrprop} and the relation in the previous section are the two key ingredients.\\
In Section \ref{section generating}, after an introductory example and a brief introduction on the general connection between duality and intertwining relations, we establish an intertwining between the discrete and the continuum processes. This intertwining relation is then used to produce all the self-duality functions for the Brownian Energy Process.

\section{Setting}\label{section setting}
We start defining what we mean by \emph{duality} of stochastic processes. Then, we introduce a general class of Markov interacting particle systems with associated interacting diffusion systems.
\subsection{Duality}\label{section setting duality}
Given two state spaces $\Omega$ and $\hat \Omega$ and \emph{two stochastic processes} $\{\xi(t),\ t\geq 0\}$ and $\{\eta(t),\ t\geq 0\}$ evolving on them, we say that they are \emph{dual} with duality function $D: \hat \Omega \times \Omega \to \R$ (where $D$ is a measurable function)
if, for all $t>0$, $\xi \in \hat \Omega$ and $\eta \in \Omega$,
we have the so-called duality relation
\be\label{dualrel}
\hat \E_\xi D(\xi(t),\eta)= \E_\eta D(\xi,\eta(t))\ .
\ee
If the laws of the two processes coincide, we speak about \emph{self-duality}.\\
More generally, we say that \emph{two semigroups} $\{S(t),\ t \geq 0\}$
and $\{ \hat{S}(t),\ t\geq 0\}$
are dual
with duality function $D$ if, for all $t \geq 0$,
\[
(\hat{S}(t))_{\eft} D= (S(t))_{\ight}D\ ,
\]
where \textquotedblleft left\textquotedblright\ (resp.\ \textquotedblleft right\textquotedblright) refers to action on the left (resp.\ right) variable.
Even more generally, we say that \emph{two operators} $L$ and $\hat{L}$ are dual to each
other
with duality function $D$ if
\be\label{opdual}
(\hat{L})_{\eft} D= (L)_{\ight}D\ .
\ee
In order not to overload notation, we use the expression $A_{\eft} D(\xi,\eta)$ for $(A D(\cdot, \eta))(\xi)$ and,
similarly, $B_{\ight} D(\xi,\eta)= (BD(\xi,\cdot))(\eta)$. We will often write $D(\xi,\eta)$
in place of $(\xi,\eta)\mapsto D(\xi,\eta)$.


\subsection{The lattice and the factorization over sites}\label{section setting lattice} The underlying geometry of all systems that we will look at consists of a  set of sites $V$ either finite or $V=\Zd$. Moreover we are given a family of transition rates $p: V \times V \to \R_+$, satisfying the following conditions: for all $x$, $y \in V$,
\begin{enumerate}
	\item[(1)] \emph{Vanishing diagonal}:  $ p(x,x)=0\ ,$
	\item[(2)] \emph{Irreducibility}:  there exist $x_1=x$, $x_2$,  $\ldots$, $x_m=y$ such that $\prod_{l=1}^{m-1} p(x_l, x_{l+1} ) > 0\ .$
\end{enumerate}
In case of infinite $V$, we further require the following:
\begin{enumerate}
	\item[(3)] \emph{Finite-range}: there exists $R> 0$ such that, for all $x,y \in V$, $p(x,y)=0$ if $|x-y|>R\ ,$
	\item[(4)] \emph{Uniform bound on total jump rate}: $\sup_{x \in V} \sum_{y \in V} p(x,y) < \infty\ .$
\end{enumerate}
Notice that when $p$ is finite-range and translation invariant, then the uniform bound on total jump rate follows automatically.

To each site $x \in V$ we associate a variable $\eta_x \in E= \N$, $\{0,\ldots,N\}$ or $\R_+$, with the interpretation of either the number of particles or the amount of energy associated to the site $x$. Configurations are denoted by $\eta \in \Omega=E^V$.

\

In all the examples that we will be discussing here, the duality functions \emph{factorize over sites}, i.e.
\be\label{factoo}
D(\xi,\eta)= \prod_{x\in V} d(\xi_x,\eta_x)\ ,\quad \eta \in E^V,\quad \xi \in \hat E^V\ .
\ee
We then call the functions $d(\xi_x,\eta_x)$ the \emph{singe-site duality functions} and further assume
\beq\label{Da}
 d(0,\cdot)\equiv 1\ .
\eeq
The above condition \eqref{Da}  is related to the fact that we want to have duality functions which make sense for infinite systems when the dual configuration has a finite total mass. A typical example is when $\eta$, $\xi \in \N^V$, where $\eta$ is an infinite configuration while $\xi$ is a finite configuration, so that in the product \eqref{factoo} there are  only a finite number of factors different from $d(0, \eta_x)$.  In this sense, the choice $d(0,\cdot)\equiv 1$ is the only sensible one for infinite systems.\\
When $V$ is finite and $E = \N$ or $\{0,\ldots,N \}$,  this condition is not necessary and e.g.\  if a  reversible product measure $\mu = \otimes_{x \in V}\ \nu$ exists, then the so-called \emph{cheap self-duality function}
$D(\xi,\eta)=\tfrac{1}{\mu(\xi)} \1{\{\xi=\eta\}} = \prod_{x \in V} \frac{1}{\nu(\xi_x)} \1\{\xi_x = \eta_x \}$ does \emph{not} satisfy \eqref{Da}.

\subsection{Interacting particle systems with factorized self-duality} \label{section setting particle}

The  class of interacting particle systems we consider is described by the following infinitesimal generator acting on local functions $f : \Omega \to \R$ as follows:
\beq \label{gzrpgen}
L f(\eta) &=& \sum_{x,y \in V} p(x,y) L_{x,y} f(\eta)\ ,\quad \eta \in \Omega\ ,
\eeq
where $L_{x,y}$, the \emph{single-edge generator}, is defined as
\beq \label{single-edge L}
L_{x,y}f(\eta)&=& u(\eta_x) v(\eta_y) (f(\eta^{x,y})-f(\eta)) + u(\eta_y)v(\eta_x)(f(\eta^{y,x})-f(\eta))\ ,\quad \eta \in \Omega\ ,
\eeq
and $\eta^{x, y}$ denotes the configuration arising from $\eta$ by removing one particle at $x$ and putting it at $y$, i.e.\  $\eta^{x,y}_x= \eta_x-1$, $\eta^{x,y}_y= \eta_y+1$, while $\eta^{x,y}_z= \eta_z$ if $z \neq x,y$. Note the conservative nature of the system and the form of the particle jump rates in \eqref{single-edge L} which depend on the number of particles in the departure and arrival site in a factorized form. Minimal requirements on the functions $u$ and $v$, namely
\begin{itemize}
	\item[(i)] $u(0)=0$, $u(1)=1$ and $u(n) > 0$ for all $n > 0\ ,$
	\item[(ii)] $v(0)\neq 0$ and $v(N)=0$ if $E=\{0,\ldots,N \}$ and in all other cases $v(n)>0\ $,
\end{itemize}
guarantee the existence of a one-parameter family of \emph{stationary} (actually \emph{reversible}) \emph{product measures} $\{\otimes\ \nu_\lambda,\ \lambda > 0 \}$ with marginals $\nu_\lambda$ given by
\beq \label{marginals}
\nu_\lambda(n) &=&  \varphi(n) \frac{\lambda^n}{n!} \frac{1}{Z_\lambda}\ ,\quad n \in \N\ \text{or}\ \{0,\ldots,N\}\ ,
\eeq
for all $\lambda > 0$ for which the normalizing constant $Z_\lambda < \infty$ and with  $\varphi(n)=n! \prod_{m=1}^{n}\frac{v(m-1)}{u(m)}$.

\subsection{Basic examples}\label{secex}

We recall here the basic examples of self-dual interacting particle systems and corresponding factorized self-duality functions known in literature (cf.\ e.g.\ \cite{gkrv}).

\begin{itemize}
	\item[(I)]\textbf{Independent random walkers ($\IRW$)}
	\begin{itemize}
		\item $E=\N\ ,$
		\item $u(n)=n\ ,$ $v(n)=1\ ,$
		\item $\nu_\lambda \sim \text{Poisson}(\lambda)\ ,$  $\nu_\lambda(n)=\tfrac{\lambda^n}{n!} e^{-\lambda}\ ,\ \lambda > 0\ ,$
		\item $d(k,n)= \frac{n!}{(n-k)!} \1\{k \leq n \}\ .$
	\end{itemize}
	 \item[(II)] \textbf{Symmetric inclusion process ($\SIP(\alpha)$, $\alpha > 0$)}
	 \begin{itemize}
	 	\item $E=\N\ ,$
	 	\item $u(n)=n\ ,$\ $v(n)=\alpha+n\ ,$
	 	\item $\nu_\lambda \sim \text{Gamma}_{\text{d}}(\alpha,\lambda)\ ,$ $\nu_\lambda(n)= \tfrac{\Gamma(\alpha+n)}{\Gamma(\alpha)} \tfrac{\lambda^n}{n!}(1-\lambda)^\alpha\ ,\ \lambda \in (0,1)\ ,$
	 	\item $d(k,n)= \tfrac{\Gamma(\alpha)}{\Gamma(\alpha+k)}\tfrac{n!}{(n-k)!} \1\{k \leq n \}\ .$
	 \end{itemize}
 \item[(III)] \textbf{Symmetric exclusion process ($\SEP(\gamma)$, $\gamma \in \N$)}
 \begin{itemize}
 	\item $E=\{0,\ldots,N \}\ ,$
 	\item $u(n)=n\ ,$\ $v(n)=\gamma-n\ ,$
 	\item $\nu_\lambda \sim \text{Binomial}(\gamma,\tfrac{\lambda}{1+\lambda})$, $\nu_\lambda(n)= \tfrac{\gamma!}{(\gamma-n)!}\tfrac{\lambda^n}{n!} \left(\tfrac{1}{1+\lambda}\right)^\gamma\ ,\ \lambda>0\ ,$
 	\item $d(k,n)= \tfrac{(\gamma-k)!}{\gamma!}\tfrac{n!}{(n-k)!} \1\{k \leq n \}\ .$
 \end{itemize}
\end{itemize}
\label{section factorizable processes}

In the following theorem we show that the only processes with generator of the type \eqref{gzrpgen} which have non-trivial factorized self-duality functions
are of one of the types described in the examples above, i.e.\ $\IRW$, $\SIP$ or $\SEP$. Here by \textquotedblleft non-trivial\textquotedblright\ we mean that the first single-site self-duality function
$d(1,n)$ is not a constant (as a function of $n$).
\bt\label{gzrprop}
Assume that the process with generator \eqref{single-edge L} is self-dual with factorized self-duality function $D(\xi,\eta)=\prod d(\xi_x,\eta_x)$ in the form \eqref{factoo}  with $d(0,\cdot)\equiv 1$ as in \eqref{Da}.
	If   $d(1,n)$ is not constant as a function of $n$,
	then
	\beq\label{didi}
	u(n) &=& n\  \nonumber
	\\
	v(n) &=& v(0)+ (v(1)-v(0))n\ ,
	\eeq
	and the first single-site self-duality function is of the form
	\beq \label{singdual1}
	d(1,n)=a + b n\ ,
	\eeq
	for some $a \in \R$ and $b \neq 0$.
	
\et
\bpr
Using the self-duality relation for $\xi_x = 1$ and no particles elsewhere, together with $u(0)=0$, we obtain the identity
\beq \nonumber
&& u(\eta_x)v(\eta_y) (d(1,\eta_x-1)- d(1,\eta_x)) +  u(\eta_y)v(\eta_x) (d(1,\eta_x+1)-d(1,\eta_x))\\
\label{ideg}
&&=\ p(x,y) u(1)v(0)( d(1,\eta_y)-d(1,\eta_x))
\ .
\eeq
Setting $\eta_x=\eta_y=n \geq 1$, this yields anytime $u(n)v(n) \neq  0$
\beq\label{harmin}
d(1,n+1)+d(1,n-1)-2 d(1,n)=0\ ,
\eeq
from which we derive $d(1,n)= a+ bn$. Because $d(1,n)$ is  not constant as a function of $n$, we must have $b\not=0$. Inserting
$d(1,n)= a+ bn$ in \eqref{ideg} we obtain
\[
u(\eta_x)v(\eta_y)-u(\eta_y)v(\eta_x)= -u(1)v(0)(\eta_y-\eta_x)\ ,
\]
from which, by setting $\eta_x=n$ and  $\eta_y=0$ we obtain the first in \eqref{didi}, while via $\eta_x=n$ and $\eta_y=1$ we get the second condition.
\epr
\begin{remark}
	 More generally, if we replace \eqref{Da} with $d(0,n)\neq 0 $ in the above statement,  we analogously obtain \eqref{didi} and
	\beq\nonumber \label{relation2}
	d(0,n) &=& c^n\\
	d(1,n) &=& (a+bn) \cdot c^n\ ,
	\eeq
	for some constants $a, b, c \in \R$, $b, c \neq 0$.
	
\end{remark}

\subsection{Interacting diffusion systems as scaling limits}\label{section setting diffusion}
Interacting diffusion systems arise as scaling limits of the particle systems in Section \ref{secex} above (cf.\ \cite{gkrv}). More in details, by \textquotedblleft scaling limit\textquotedblright\ we refer to the limit process of the particle systems $\{\tfrac{1}{N}\eta^N(t), \ t \geq 0 \}_{N \in \N}$, where the initial conditions $\tfrac{1}{N}\eta^N(0)= \tfrac{1}{N}(\lfloor z_x N \rfloor )_{x \in V}$ converge to some  $z \in E^V$, with $E = \R_+$.

In case of $\IRW$, one obtains a deterministic (hence degenerate diffusion) process $\{z(t),\ t \geq 0 \}$ whose evolution is described by a first-order differential operator. In case of $\SIP(\alpha)$, the scaling limit is a proper Markov process of interacting diffusions known as Brownian Energy Process ($\BEP(\alpha)$) (cf.\ \cite{gkrv}).
For the $\SEP(\gamma)$, this limit cannot be taken in the sense of Markov processes, but we can extend the SEP generator to
a larger class of functions defined on a larger configuration space and take the many-particle limit.
The limiting second-order differential operator is then not a Markov generator, but still a second order differential operator. We will explain this more in detail below.

The limiting  differential operators in the case of $\IRW$ and $\SIP(\alpha)$ can be described as acting on smooth functions $f : E^V \to \R$ as follows:
\beq \label{generators diffusions}
\caL f(z) &=& \sum_{x,y \in V} p(x,y) \caL_{x,y} f(z)\ ,\quad z \in E^V\ ,
\eeq
with single-edge generators $\caL_{x,y}$ given, respectively,  by
\beq \nonumber
\caL_{x,y} f(z) &=& [-(z_x - z_y) (\partial_x-\partial_y) ] f(z)\ ,\quad z \in E^V\ ,
\eeq
and
\beq \nonumber
\caL_{x,y} f(z) &=&  [-\alpha (z_x - z_y) (\partial_x-\partial_y) + z_x z_y (\partial_x-\partial_y)^2] f(z)\ ,\quad z \in E^V\ .
\eeq
For the $\SEP(\gamma)$ we proceed as follows. For each $N \in \N$, consider the operator $L^N$ working on functions $f : (\N/N)^V \to \R$ as
\beq \label{extendedsep}
L^N f(\tfrac{1}{N}\eta) &=& \sum_{x,y\in V} p(x,y) L^N_{x,y}f(\tfrac{1}{N}\eta)\ ,\quad  \eta \in \N^V\ ,
\eeq
where
\be \nonumber
L^N_{x,y} f(\tfrac{1}{N}\eta)= \eta_x(\gamma-\eta_y) (f(\tfrac{1}{N}\eta^{x,y})-f(\tfrac{1}{N}\eta)) +\eta_y(\gamma-\eta_x) (f(\tfrac{1}{N}\eta^{y,x})-f(\tfrac{1}{N}\eta))\ ,\quad \eta \in \N^V\ .
\ee
This operator is not a Markov generator anymore, because the factors $\eta_x(\gamma-\eta_y)$ can become negative.
With this operator, we  consider the limit
\beq \nonumber
\lim_{N \to \infty} (L^N f)(\tfrac{1}{N}\eta^N)\ ,
\eeq
where $\eta^N= (\lfloor N z_x\rfloor)_{x \in V}$ and $f : E^V \to \R$ is a smooth function.
This then gives the differential operator $\caL$ which is the analogue of \eqref{generators diffusions} in the context of
$\SEP(\gamma)$. This  differential operator $\caL$, with single-edge operators
\beq \label{sepdifop}
\caL_{x,y}f(z) &=& [-\gamma (z_x-z_y)(\partial_x-\partial_y) - z_x z_y(\partial_x-\partial_y)^2] f(z)\ ,\quad z \in E^V\ ,
\eeq
does not generate a Markov process but it is still useful because, as we will see
in Section \ref{section generating} below, via generating functions, it is intertwined with the operator \eqref{extendedsep} for the choice $N=1$.

\

 Naturally, as we can see for the case of $\SIP(\alpha)$ and $\BEP(\alpha)$, when going to the scaling limit, some properties concerning stationary measures and duality pass to the limit. Indeed, $\BEP(\alpha)$ admits a one-parameter family of stationary product measures $\{\otimes\ \nu_\lambda,\ \lambda > 0 \}$, where $\nu_\lambda \sim \text{Gamma}(\alpha,\lambda)$, namely
\beq \label{gamma distribution}
\nu_\lambda(dz) &=& z^{\alpha-1} e^{-\lambda z} \frac{\lambda^\alpha}{\Gamma(\alpha)} dz\ ,
\eeq
and is dual to $\SIP(\alpha)$ with factorized duality function $D(\xi,z)=\prod_{x \in V}d(\xi_x, z_x)$ given by
\beq \nonumber
d(k,z) &=& z^k \frac{\Gamma(\alpha)}{\Gamma(\alpha+k)}\ ,\quad k \in \N\ ,\quad z \in \R_+\ .
\eeq
After noting that property \eqref{Da}  holds also in this situation,  we show that the first single-site duality functions $d(1,x)$
between $\SIP(\alpha)$ and $\BEP(\alpha)$  are affine functions of $z \in \R_+$, as we found earlier for single-site self-dualities in Theorem \ref{gzrprop}.
\begin{proposition}
	Assume that $\SIP(\alpha)$ and $\BEP(\alpha)$'s single-edge generators  are dual with factorized duality function $D(\xi,z)=\prod d(\xi_x, z_x)$ with $d(0,\cdot)\equiv 1$ as in \eqref{factoo}--\eqref{Da}.  Then
	\beq\label{ghgh}
	d(1,z) &=& a + b z\ ,
	\eeq
	for some $a, b \in \R$.
\end{proposition}
\begin{proof}
	The duality relation for $\xi_x = 1$ and no particles elsewhere, by using \eqref{Da}, reads
	\beq \nonumber
	d(1,z_y) - d(1, z_x) &=&
	-\alpha (z_x - z_y) \partial_x d(1,z_x) + z_x z_y \partial_x^2 d(1,z_x)\ .
	\eeq
	If we set $z_x = z_y = z$, then $z^2 \frac{d^2}{dz^2} d(1,z)=0$ leads to \eqref{ghgh} as unique solution.
\end{proof}

%

\section{Relation between duality function and stationary product measure}\label{section general relation}

In the examples of duality that we have encountered in the previous section, we have a universal relation between
the stationary product measures and the duality functions.
Given $\{\eta(t),\ t \geq 0 \}$, if there is a dual process $\{\xi(t),\ t \geq 0 \}$ with factorized duality functions
$D(\xi,\eta)=\prod d(\xi_x,\eta_x)$ as in \eqref{factoo}--\eqref{Da} and stationary product measures $\{\mu_\lambda= \otimes\ \nu_\lambda,\ \lambda > 0\}$, then there is a relation between these measures
and these  functions, namely there exists a function $\theta(\lambda)$ such that
\beq\label{reldual}
\int D(\xi,\eta) \mu_\lambda (d\eta)\ =\ \prod_{x \in V} \int d(\xi_x,\eta_x) \nu_\lambda(d\eta_x) &=& \theta(\lambda)^{|\xi|}\ .
\eeq
This function $\theta(\lambda)$ is then simply the expectation of the first single-site duality function, i.e.\
\beq \nonumber
\theta(\lambda) &=& \int d(1,\eta_x) \nu_\lambda( d\eta_x)\ .
\eeq
In the examples of Section \ref{secex}, we have $\theta(\lambda)=\lambda$ for $\IRW$, $\theta(\lambda)=\tfrac{\lambda}{1-\lambda}$ for $\SIP(\alpha)$ and
$\theta(\lambda)=\tfrac{\lambda}{1+\lambda}$ for $\SEP(\gamma)$.

In this section we first  investigate under which general conditions this relation holds, and further use it in Sections \ref{section translation}--\ref{section non-translation} as a criterion of characterization of all extremal measures. The reader shall refer to Sections \ref{section setting duality}--\ref{section setting lattice} for the  general  setting in which these results hold.

Later on, we will see that this relation \eqref{reldual} is actually a characterizing property of the  factorized duality functions, meaning that
all duality functions are determined once the first single-site duality function is fixed.

\subsection{Finite case}\label{section finite}
We start with the simplest situation in which $V$ is a finite set.

First, we assume that the total number of
particles/the total energy of the dual process is the only conserved quantity.
More precisely, we assume the following property, which we refer to as \textit{harmonic triviality} of the dual system:
\begin{enumerate}
\item[\HT]
If $H:\hat E^V\to\R$ is harmonic, i.e.\ such that, for all $t>0$, $$\hat \E_\xi H(\xi(t))=H(\xi)\ ,$$ then
$H(\xi)$ is only a function of $|\xi|:=\sum_{x\in V} \xi_x$.
\end{enumerate}
Moreover, let us assume the factorized form of the duality function $D(\xi, \eta)$ as in \eqref{factoo}--\eqref{Da}.
Of the single-site functions $d$ we may require the following additional property:
\ben
\item[\Db] The function $d$ is \textit{measure determining}, i.e.\ for two probability measures
$\nu_\ast$, $\nu_\star$ on $E$ such that for all $x\in V$ and $\xi_x\in \hat E$
\[
\int_E d(\xi_x, \eta_x)\nu_\ast(d\eta_x)=\int_E d(\xi_x,\eta_x)\nu_\star(d\eta_x)<\infty\ ,
\]
it follows that $\nu_\ast=\nu_\star$.
\een

Then we have the following.
\bt\label{finitethm}
Assume that $\{\eta(t),\ t\geq 0\}$ and $\{\xi(t),\ t \geq 0\}$ are  dual as in \eqref{dualrel} with factorized duality function \eqref{factoo}
satisfying condition \eqref{Da}. Moreover, assume that \HT\ holds and that $\mu$ is a probability measure
on $\Omega$. We distinguish two cases:
\begin{enumerate}
\item[(1)] \emph{Interacting particle system case.} If $\hat E$ is a subset of $\N$, then we assume that the self-duality functions $D(\xi, \cdot)$ are $\mu$-integrable for all $\xi \in \hat \Omega$.
\item[(2)] \emph{Interacting diffusion case.} If $\hat E=\R_+$, then we assume the following integrability condition: for each $\varepsilon > 0$, there exists a $\mu$-integrable function $f_\varepsilon$ such that
\beq \label{unifintegrability}
\sup_{\xi \in \hat \Omega,\ |\xi|=\varepsilon} |D(\xi,\eta)| &\leq& f_\varepsilon(\eta)\ ,\quad \eta \in \Omega\ .
\eeq
\end{enumerate}
Then
\ben
\item[(a)]
$\mu$ is a stationary product measure for the process $\{\eta(t):t\geq 0\}$
\een
implies
\ben
\item[(b)] For all $\xi\in \Omega$ and for all $x\in V$,
we have
\beq\label{basicrelfi}
\int D(\xi,\eta) \mu (d\eta)&=& \left(\int D(\delta_x, \eta) \mu(d\eta)\right)^{|\xi|}\ ,
\eeq
\een
where $\delta_x$ denotes the configuration with a single particle at $x \in V$ and no particles elsewhere. Moreover, if condition \Db\ holds, the two statements (a) and (b) are equivalent.
\et
\bpr
First assume that $\mu$ is a stationary product measure.
Define $H(\xi)=\int D(\xi,\eta) \mu(d\eta)$. By $\mu$-integrability in the interacting particle system case (resp.\ \eqref{unifintegrability} in the interacting diffusion case),  self-duality and invariance of $\mu$ we have
\beq
\E_\xi H(\xi(t))\ =\ \int \E_\xi D(\xi(t),\eta) \mu(d\eta)\ =\ \int \E_\eta D(\xi,\eta(t)) \mu(d\eta)\nonumber\
 =\ \int  D(\xi,\eta) \mu(d\eta)
=H(\xi) \nonumber\ .
\eeq
Therefore by \HT\ we conclude that $H(\xi)=\psi(|\xi|)$. By using $d(0,\cdot)\equiv 1$ and
the factorization of the duality functions, we have that $\psi(0)=1$. For the particle case, we obtain that
\beq \nonumber
\int D(\delta_x, \eta) \mu(d\eta)&=& \psi(1)\ .
\eeq
In particular, we obtain that the l.h.s.\  does not depend on $x$.
Next, for $n\geq 2$, put
\beq \nonumber
\int D(n\delta_x, \eta) \mu(d\eta)&=& \psi(n)\ ,
\eeq
then we have for $x\not=y \in V$, using the factorized duality function and the product form of the measure,
\beq \nonumber
\psi(n)\ =\ \int D(n\delta_x, \eta) \mu(d\eta)\ =\ \int D(\delta_y,\eta) D((n-1)\delta_x, \eta) \mu(d\eta)\
=\ \psi(1) \psi(n-1)\ ,
\eeq
from which it follows that $\psi(n)= \psi(1)^n$. Via an analogous reasoning that uses the factorization of $D(\xi,\eta)$ and the product form of $\mu$, for the diffusion case we obtain, for all $\varepsilon, \rho \geq 0$,
\beq \nonumber
\psi(\varepsilon+\rho) &=& \psi(\varepsilon) \psi(\rho)\ ,
\eeq
and hence, by measurability of $\psi(\varepsilon)$, we get $\psi(\varepsilon)=\psi(1)^\varepsilon$.

To prove the other implication, put
\beq \nonumber
\int D(\delta_x, \eta) \mu(d\eta)&=&\kappa\ .
\eeq
We then have by assumption
\beq \nonumber
\int D(\xi,\eta)\mu (d\eta)&=&\kappa^{|\xi|}\ ,
\eeq
and so it follows that $\mu$ is stationary by self-duality, $\mu$-integrability, the conservation of the number of particles and the measure-determining property.  Indeed,
\beq \nonumber
\int \E_\eta D(\xi,\eta(t)) \mu (d\eta)\ =\  \int \E_\xi D(\xi(t),\eta) \mu (d\eta)\ =\ \E_\xi (\kappa^{|\xi(t)|})\ =\ \kappa^{|\xi|}=\int  D(\xi,\eta) \mu (d\eta)\ .
\eeq
From the factorized form of $D(\xi,\eta)$, \eqref{basicrelfi} implies that for all $x\in V$ and $\xi_x\in \hat E$
\beq \nonumber
\int d(\xi_x,\eta_x)\mu (d\eta) &=& \kappa^{\xi_x}\ ,
\eeq
and also
\beq \nonumber
\int D(\xi,\eta)\mu (d\eta)\ =\  \kappa^{|\xi|}\ =\ \prod_{x\in V}\kappa^{\xi_x}\ =\ \prod_{x\in V}\int d(\xi_x,\eta_x)\mu (d\eta)\ ,
\eeq
therefore $\mu$ is a product measure by the fact that $d$ is measure determining.
\epr

\subsection{Infinite case}\label{section infinite}
If $V=\Zd$, then one needs essentially two extra conditions to state an analogous result in which a general relation between duality functions
and corresponding stationary measures can
be derived.

In this section we will assume that the dual process
is a discrete particle system, i.e.\ $\hat E$ is a subset of $\N$,  in which
the number of particles is conserved. In this case we need an additional property ensuring that
for the dynamics of a finite number of particles there are no  bounded harmonic functions other than those depending on the total number of particles.
Therefore, we introduce the condition of existence of a successful coupling for the discrete
dual process with a finite number of particles.
This is defined below.
\bd
We say  that the discrete dual process $\{\xi(t),\ t\geq 0\}$ has the \emph{successful coupling property}
when the following holds:
if we start with $n$ particles then there exists a labeling such that for the
corresponding labeled process
$\{X_1(t), \ldots, X_n(t),\ t\geq 0\}$ there exists a successful coupling.
This means that for every two initial positions ${\bf x}=(x_1,\ldots, x_n)$ and ${\bf y}=(y_1,\ldots, y_n)$, there
exists a coupling with path space measure $\pee_{\bf x,y}$ such that the coupling time
\[
\tau=\inf \{s>0:\ {\bf X}(t)={\bf Y}(t),\ \forall t\geq s \}
\]
is finite $\pee_{\bf x,y}$ almost surely.
\ed
Notice that the successful coupling property
is the most common way to prove the following property (cf.\ \cite{ligg}), which is the analogue of \HT, referred here to as \textit{bounded harmonic triviality} of the dual process:
\begin{enumerate}
\item[\BHT]
If
 $H$ is a bounded harmonic function,
then $H(\xi)=\psi(|\xi|)$ for some bounded $\psi: \hat E \to \R$.
\end{enumerate}
\begin{remark}
	The condition of the existence of a successful coupling (and the consequent bounded harmonic triviality) is quite natural
	in the context of interacting particle systems, where we have that a finite number of walkers behave as independent walkers,
	except when they are close and interact. Therefore, the successful coupling needed is a variation of the Ornstein coupling
	of independent walkers, see e.g.\ \cite{demasi}, \cite{ferrari}, \cite{kuoch}.
\end{remark}

Furthermore, we need a form of uniform $\mu$-integrability of the duality functions which we introduce below and call \textit{uniform domination property} of $D$ w.r.t.\ $\mu$ (note the analogy with condition \eqref{unifintegrability}):

\begin{enumerate}
\item[\UD]
\label{unifintD}
Given $\mu$ a probability measure on $\Omega$, the duality functions $\{D(\xi,\cdot),\ |\xi|=n\}$ are uniformly $\mu$-integrable, i.e.\ for all $n \in \N$
there exists a function $f_n$ such that $f_n$ is $\mu$-integrable and such that for
all $\eta\in \Omega$
\be\label{domi}
\sup_{\xi \in \hat \Omega,\ |\xi|=n} |D(\xi, \eta)|\leq f_n(\eta)\ .
\ee
\end{enumerate}

Under these conditions, the following result holds, whose proof resembles that of Theorem \ref{finitethm}.
\begin{theorem}\label{infinitethm}
Assume as in \eqref{dualrel} that $\{\eta(t),\ t\geq 0\}$ is dual to the discrete process $\{\xi(t),\ t \geq 0 \}$ with factorized duality function as in \eqref{factoo}--\eqref{Da}. Moreover, assume \BHT\ in place of \HT\ for the dual process and that $\mu$ is a probability measure
on $\Omega$ such that \UD\ holds. Then the same conclusions as in Theorem \ref{finitethm} follow, where \eqref{basicrelfi} holds for all finite configurations $\xi \in \hat \Omega$.
\end{theorem}

\subsubsection{Translation invariant case}\label{section translation}

In this section, we show that under the assumption of factorized duality functions, minimal ergodicity assumptions on a stationary probability measure $\mu$ on $\Omega$ are needed to ensure  \eqref{basicrelfi} and, as a consequence, that $\mu$ is product measure.

Here we restrict to the case $V=\Zd$ because we will use spatial ergodicity.

\bt\label{transthm}
In the setting of Theorem \ref{infinitethm} with $D(\xi,\eta)$ factorized duality function and $\mu$ probability measure on $\Omega$, if $\mu$ is a translation-invariant and ergodic (under translations)
stationary measure for $\{\eta(t),\ t \geq 0 \}$, then we have \eqref{basicrelfi} for all finite configurations $\xi$; as a consequence, $\mu$ is a product measure.
\et
\bpr
To start, let us consider a configuration $\xi=\sum_{i=1}^n \delta_{x_i}$.
By bounded harmonic triviality \BHT, combined with the bound \eqref{domi} for all such configurations,
$\int D(\xi, \eta) \mu(d\eta)$ is only depending on $n$ and, therefore,
we can replace $\xi$ by $\sum_{i=1}^{n-1} \delta_{x_i} +\delta_y$, where $y$ is
arbitrary in $\Zd$. Let us call $B_N=[-N,N]^d\cap\Zd$. Fix $N_0$ such that $B_{N_0}$ contains
all the points $x_1,\ldots x_{n-1}$. For $y$ outside $B_{N_0}$, by the factorization property,
$D(\sum_{i=1}^{n-1} \delta_{x_i}+\delta_y, \eta)=D(\sum_{i=1}^{n-1} \delta_{x_i}, \eta) D(\delta_y, \eta)$.
By the Birkhoff ergodic theorem,
we have that
\[
\frac{1}{(2N+1)^d}\sum_{y\in B_N} D(\delta_y, \eta)\to \int D(\delta_0, \eta)\mu(d\eta)
\]
$\mu$-a.s.\ as $N\to\infty$.
Using this, together with  \eqref{domi},
we have
\beq\label{bombi}
&&\int D(\xi, \eta) \mu(d\eta)
\nonumber\\
&&=\lim_{N\to\infty}\frac{1}{(2N+1)^d}\sum_{y\in B_N\setminus B_{N_0}} \int D(\sum_{i=1}^{n-1} \delta_{x_i}, \eta) D(\delta_y, \eta)\mu(d\eta)
\nonumber\\
&&= \int D(\sum_{i=1}^{n-1} \delta_{x_i}, \eta) \mu(d\eta) \int D(\delta_0, \eta)\mu(d\eta)
\nonumber\\
\eeq
Iterating this argument gives \eqref{basicrelfi}.
\epr

\br
 As follows clearly from the proof, the condition of factorization of the duality function can be replaced by the weaker condition of
\[
\lim_{|y|\to\infty} (D(\delta_{x_1}+\ldots+ \delta_{x_n} + \delta_y, \eta)- D(\delta_{x_1}+\ldots+ \delta_{x_n},\eta) D( \delta_y, \eta))=0
\]
for all $\eta$, $x_1,\ldots,x_n$. We note that this approximate factorization of the duality function leads to \eqref{basicrelfi}, though $\mu$ is not necessarily a product measure.
\er
\subsubsection{Non-translation invariant case}\label{section non-translation}
We continue here with $V=\Zd$ but drop the
assumption of translation invariance. Indeed,
equality \eqref{basicrelfi} is also valid in contexts where one cannot rely on translation invariance.
Examples include spatially inhomogeneous $\SIP(\boldsymbol \alpha)$ and $\SEP(\boldsymbol \seppar)$, where the parameters   $\boldsymbol \alpha = (\alpha_x)_{x \in V}$ and $\boldsymbol \gamma = (\gamma_x)_{x \in V}$ in Section \ref{section setting particle} may depend on the site accordingly (cf.\ e.g.\ \cite{rs}). Also in this inhomogeneous setting the self-duality functions factorize over  sites and the stationary measures are in product form, with site-dependent single-site duality functions, resp. site-dependent marginals. We will show that the relation  \eqref{basicrelfi} between
the self-duality functions and any ergodic stationary measure still holds, and as a consequence this ergodic stationary measures is in fact a product measure. The idea is that the averaging over space w.r.t.\ $\mu$, used in the proof of Theorem \ref{transthm} above, can be replaced by a time average.

If we start with a single dual particle, the dual process is a continuous-time random walk on $V$, for which we
denote by $p(t; x,y)$ the transition probability to go from $x$ to $y$ in time $t>0$.
A basic assumption will then be
\be\label{zeropt}
\lim_{t\to\infty} p(t; x,y)=0
\ee
for all $x, y \in V$.
\bt\label{tempergthm}
In the setting of Theorem \ref{infinitethm} with $D(\xi,\eta)$ duality function and $\mu$ probability measure on $\Omega$, if $\mu$ is an ergodic stationary measure for the process $\{\eta(t),\ t \geq 0 \}$ and \eqref{zeropt} holds for the dual particle, then we have \eqref{basicrelfi} for all finite configurations $\xi$; as a consequence, $\mu$ is a product measure.
\et
\bpr
The idea is to replace the spatial average in the proof of Theorem \ref{transthm} by a Cesaro average over time, which we can deal by combining
assumption \eqref{zeropt} with the assumed temporal ergodicity.

Fix $x_1,\ldots,x_n\in V$, $y\in V$.
Define
\beq \nonumber
H_{n+1}(x_1,\ldots,x_n,y)&=& \int D(\delta_{x_1}+ \ldots+\delta_{x_n}+ \delta_y,\eta) \mu(d\eta)
\\
\nonumber
H_1(y)&=&  \int D(\delta_y,\eta) \mu(d\eta)
\\ \nonumber
H_{n}(x_1,\ldots,x_n)&=& \int D(\delta_{x_1}+ \ldots+\delta_{x_n},\eta) \mu(d\eta)\ .
\eeq

It is sufficient to obtain that $H_{n+1}(x_1,\ldots,x_n,y)= H_1(y)H_{n}(x_1,\ldots,x_n)$.
We already
know by the bounded harmonic triviality that $H_{n}$ only depends on $n$ and not on the given locations $x_1,\ldots,x_n$.
Therefore, we have
\beq \nonumber
H_{n+1}(x_1,\ldots,x_n,y) &=& \sum_{z} p(t; y,z) H_{n+1}(x_1,\ldots,x_n,z)\ .
\eeq
By assumption \eqref{zeropt}, this implies
\beq
&&H_{n+1}(x_1,\ldots,x_n,y)\nonumber\\ &=& \lim_{T\to\infty}\frac1T \int_0^T dt \sum_{z\not \in \{x_1,\ldots,x_n\}} p(t; y,z) H_{n+1}(x_1,\ldots,x_n,z)
\nonumber\\
&=&
\lim_{T\to\infty}\frac1T \int_0^T dt \sum_{z\not \in \{x_1,\ldots,x_n\}} p(t; y,z) \int D(\delta_{x_1}+ \ldots+\delta_{x_n},\eta)D( \delta_y,\eta) \mu(d\eta)
\nonumber\\
&=& \lim_{T\to\infty}\frac1T \int_0^T dt \sum_{z} p(t; y,z) \int D(\delta_{x_1}+ \ldots+\delta_{x_n},\eta)D( \delta_y,\eta) \mu(d\eta)
\nonumber\\
&=&
\lim_{T\to\infty}\frac1T\int_0^T dt  \int D(\delta_{x_1}+ \ldots+\delta_{x_n},\eta)\E_\eta D( \delta_y,\eta(t)) \mu(d\eta)
\nonumber\\
&=&
H_1(y)H_{n}(x_1,\ldots,x_n)\ , \nonumber
\eeq
where in the last step we used the assumed temporal ergodicity of $\mu$ and Birkhoff ergodic theorem.
\epr

\section{From stationary product measures to duality functions.} \label{section general construction}
As we have just illustrated in Sections \ref{section translation}--\ref{section non-translation}, relation \eqref{reldual} turns out to be useful in deriving information about the product structure of stationary ergodic measures from the knowledge of factorized duality functions. On the other side, granted some information on the stationary product measures, which follows usually from a simple detailed balance computation, up to which extent does relation \eqref{reldual} say something about the possible factorized duality functions?

In the context of conditions \eqref{factoo}--\eqref{Da} and in presence of a one-parameter family of stationary product measures $\{\mu_\lambda = \otimes\ \nu_\lambda,\ \lambda > 0 \}$, relation \eqref{reldual}  for $\xi = k  \delta_x \in \hat \Omega$ for some $k \in \N$ reads
\beq \label{reldual5}
\int_E d(k,\eta_x) \nu_\lambda(d\eta_x)\ =\ \left(\int_E d(1,\eta_x) \nu_\lambda(d\eta_x) \right)^k &=& \theta(\lambda)^k\ .
\eeq
As a consequence, knowing the first single-site duality function $d(1,\cdot)$ and the explicit expression of the marginal $\nu_\lambda$ is enough  to recover the l.h.s.\ in \eqref{reldual5}. However,  rather than obtaining $d(k,\cdot)$, at this stage the l.h.s.\ has still the form of an \textquotedblleft integral transform\textquotedblright-type of expression for $d(k,\cdot)$.

In the next two subsections,  we show
how to recover $d(k,\eta_x)$ from \eqref{reldual5} and the knowledge of $\theta(\lambda)$.
This then leads to the characterization of all possible factorized (self-)duality functions.


\subsection{Particle systems and orthogonal polynomial self-duality functions}\label{section particle systems duality}
Going back to the interacting particle systems introduced in Section \ref{section setting particle} with infinitesimal generator \eqref{gzrpgen} and product stationary measures with marginals \eqref{marginals}, the integral relation \eqref{reldual5} rewrites, for each $k \in \N$ and $\lambda > 0$ for which $Z_\lambda < \infty$,
as
\beq \nonumber
\sum_{n \in \N} d(k,n) \nu_\lambda(n)\ =\ \sum_{n \in \N} d(k,n) \varphi(n) \frac{\lambda^n}{n!} \frac{1}{Z_\lambda} &=& \theta(\lambda)^k\ ,
\eeq
where $\varphi(n)= n! \prod_{m=1}^n \tfrac{v(m-1)}{u(m)}$. Now, if we interpret
\beq \nonumber
\sum_{n \in \N} d(k,n) \varphi(n) \frac{\lambda^n}{n!}
\eeq
as the \emph{Taylor series expansion} around $\lambda=0$ of the function $\theta(\lambda)^k Z_\lambda$, we can  re-obtain the explicit formula of $d(k,n) \varphi(n)$ as its $n$-th order derivative evaluated at $\lambda=0$, namely
\beq \nonumber
d(k,n) \varphi(n) &=& \left(\left[ \frac{d^n}{d\lambda^n}\right]_{\lambda=0}\theta(\lambda)^k Z_\lambda \right)\ ,
\eeq
and hence, anytime $\varphi(n) > 0$,
\beq\label{dnu}
d(k,n) &=& \frac{1}{\vi(n)}\left(\left[ \frac{d^n}{d\lambda^n}\right]_{\lambda=0}\theta(\lambda)^k Z_\lambda\right)\ .
\eeq

Together with the full characterization obtained in Theorem  \ref{gzrprop} of the first single-site self-duality functions $d(1,\cdot)$ - and $\theta(\lambda)$ in turn - we obtain via this procedure a full characterization of all single-site self-duality functions. Beside recovering the \textquotedblleft classical" dualities illustrated in Section \ref{section setting particle}, we find  the single-site self-duality functions in terms of orthogonal polynomials $\{p_k(n),\ k \in \N \}$ of a discrete variable (cf.\ e.g.\ \cite{Nikiforov}) recently discovered via a different approach in \cite{chiara}. We add the observation that all these new single-site self-duality functions can be obtained  from the classical ones via a Gram-Schmidt orthogonalization procedure w.r.t.\ the correct probability measures on $\N$, namely the marginals of the associated stationary product measures (cf.\ \cite{chiara}).

%

We divide the discussion in three cases, one suitable for processes of $\IRW$-type, the other for $\SIP$ and $\SEP$ and the last one for the remaining particle systems.

\subsubsection{Independent random walkers} We recall that the $\IRW$-case corresponds to the choice of values in \eqref{didi} satisfying the relation $v(1)=v(0)$.   If we compute $\theta(\lambda)$ for the general first single-site self-duality function $d(1,n)=a+b n$ obtained in \eqref{singdual1}, we get
\beq \nonumber
\theta(\lambda) &=& \sum_n (a+b n) \frac{\lambda^n}{n!}\frac{1}{Z_\lambda}= a + b\lambda\ .
\eeq
and, in turn via relation \eqref{dnu}, we recover all functions $d(k,\cdot)$, for $k > 1$:
\beq \nonumber
d(k,n)&=&  \left( \left[\frac{d^n}{d \lambda^n} \right]_{\lambda=0} \left(a + b\lambda\right)^k e^\lambda \right) \\
&=& \sum_{r=0}^n \binom{n}{r} k (k-1) \cdots (k-r+1) b^r a^{k-r}\ .\nonumber
\eeq
In case $a=0$, $d(k,n) = 0$ for $n < k$, while for $n \geq k$, in the summation all terms   but the one corresponding to $r=k$ vanish, thus $d(k,n)= \frac{n!}{(n-k)!} b^k$. In case $a \neq 0$,
\beq \label{single-site irw}
d(k,n) &=& a^k \sum_{r=0}^{\min(k,n)} \binom{n}{r} \binom{k}{r} r!  \left( \frac{b}{a}\right)^r\ =\ a^k\ \pFq{0}{2}{-k,-n}{-}{\frac{b}{a}}\ .
\eeq
In conclusion, for the choice $a \cdot b < 0$,
\beq \nonumber
d(k,n) &=& a^k\  C_k(n; -\tfrac{a}{b})\ ,
\eeq
where $\{C_k(n; \mu),\ k \in \N\}$ are the Poisson-Charlier polynomials - orthogonal polynomials w.r.t.\ the Poisson distribution of parameter $\mu > 0$ (cf.\ \cite{Nikiforov}).
\subsubsection{Inclusion and exclusion processes} For $\SIP$ and $\SEP$ we are in the  case $v(1) \neq v(0)$, and hence we abbreviate
\beq \nonumber
\sigma\ =\ v(1)-v(0)\ ,\quad \beta\ =\ v(0)\ ,
\eeq
where for $\SIP(\alpha)$ we choose $\sigma=1$ and $\beta=\alpha$, while for  $\SEP(\gamma)$ we set $\sigma=-1$ and $\beta=\gamma$.
If we compute $\theta(\lambda)$  for $d(1,n)=a+bn$ in \eqref{singdual1}, we have
\beq
\nonumber
\theta(\lambda)&=&   a + b \beta \lambda\left(1-\sigma \lambda \right)^{-1}\ =\ \frac{a+\left(b \beta  - a \sigma \right)\lambda}{1-\sigma \lambda}\ .
\eeq
By  applying formula \eqref{dnu}, we obtain all functions $d(k,n)$ for $k > 1$ as follows:
\beq
\nonumber
d(k,n)&=&\nonumber \frac{1}{\vi(n)}\left(\left[ \frac{d^n}{d\lambda^n}\right]_{\lambda=0}\left(a+\left(b\beta  - a \sigma \right) \lambda \right)^k \left(1-\sigma \lambda \right)^{-k-\sigma \beta} \right)\ \\
\nonumber
&=& \frac{\Gamma\left(\sigma \beta \right)}{\Gamma\left(\sigma \beta+n \right)} \sum_{r=0}^n  \binom{n}{r} \binom{k}{r} r! a^{k-r} \left(b \sigma \beta - a  \right)^{r}    \frac{\Gamma\left(\sigma \beta + n + k -r \right)}{\Gamma\left(\sigma \beta + k \right)}\ .
\eeq
In case $a=0$, clearly $d(k,n)=0$ for $k < n$, while for $n \geq k$ only the term for $r=k$  is nonzero in the summation:
\beq \label{dknclassical}
d(k,n) &=&  \frac{n!}{(n-k)!} \frac{\Gamma\left(\sigma \beta \right)}{\Gamma\left(\sigma\beta  + k\right)} \left(b \sigma\beta \right)^k\ .
\eeq
In case $a \neq 0$, by using the known relation (cf.\ \cite{Nikiforov}, p.\ 51),
\beq \nonumber \label{dknhyper}
d(k,n) &=& a^k\ \frac{\Gamma\left(\sigma\beta \right) \Gamma\left(\sigma \beta+n+k \right)}{\Gamma\left(\sigma \beta+n \right) \Gamma\left(\sigma \beta + k \right)}\ \pFq{2}{1}{-n, -k}{-n-k-\sigma \beta+1}{1-\frac{b}{a}\sigma\beta}\\
&=& a^k\ \pFq{2}{1}{-n, -k}{\sigma \beta}{\frac{b}{a} \sigma \beta}\ .
\eeq
If $\sigma=1$, $\beta = \alpha> 0$ and if $a \cdot b < 0$, we recognize in \eqref{dknhyper} the Meixner polynomials as defined in \cite{Nikiforov}, i.e.\
\beq \nonumber
d(k,n) &=& a^k\ \frac{\Gamma\left(\alpha \right)}{\Gamma\left(\alpha + k \right)}\ M_k(n;  \mu, \alpha)\ ,
\eeq
where in our case $\mu= \frac{a}{a-b \alpha}$ and $\{M_k(n;  \mu, \alpha),\ k \in \N \}$ are the Meixner polynomials - orthogonal polynomials w.r.t.\ the discrete Gamma distribution of scale parameter $\mu$ and shape parameter $\alpha$.

Furthermore, if $\sigma=-1$, $\beta =  \seppar$ with $\seppar \in \N$ and given the additional requirements  $a \cdot b < 0$ and $-\frac{a}{b} \leq  \seppar$, from the expression in \eqref{dknhyper} we have a representation of the single-site duality functions in terms of the Kravchuk polynomials as defined in \cite{Nikiforov}, i.e.\
\beq \nonumber
d(k,n) &=& a^k\ K_k(n; p, \seppar )\ \left( -\frac{1}{p} \right)^k \frac{1}{\binom{\seppar}{k}}\ ,
\eeq
where $p= -\frac{a}{b \seppar}$ in our case and $\{K_k(n; p, \seppar),\ k \in \N \}$ the Kravchuk polynomials - orthogonal polynomials w.r.t.\ the Binomial distribution of parameters $\seppar$ and $p$.

\

As a conclusion of this procedure, we note that all  factorized self-duality functions for independent random walkers, inclusion and exclusion processes satisfying \eqref{Da}  are either in the \textquotedblleft classical" form of Section \ref{section setting particle} (case $a=0$) or consist of products of rescaled versions of orthogonal polynomials (case $a \neq 0$). Other factorized self-duality functions for the systems $\IRW$, $\SIP$ and $\SEP$ do not exist.

\begin{remark}
It is interesting to note that, apart from the leading factor  $a^k$, the remaining \emph{polynomials} in the expressions of $d(k,n)$ for $a \neq 0$ are \textquotedblleft \emph{self-dual}" in the sense of the orthogonal polynomials literature, i.e.\ $p_n(k)=p_k(n)$ in our context  (cf.\ Definition 3.1, \cite{koekoek}). Henceforth, if $d(k,n)$ is interpreted as a countable matrix, the value $a \in \R$ is the only responsible for the asymmetry of $d(k,n)$: upper-triangular for $a=0$ while symmetric for $a=1$.
\end{remark}

\subsubsection{Trivial factorized self-duality}
To conclude, for the sake of completeness, we can implement the same machinery to cover  all factorized self-dualities with property \eqref{Da} for all discrete processes of  type \eqref{gzrpgen}.

Indeed, from the proof of Theorem \ref{gzrprop}, if the process is neither of the types $\IRW$, $\SIP$ and $\SEP$, then the only possible choice is $d(1,n)=a$ for some $a \in \R$, i.e.\ it is not depending on $n$. From this we get
$\theta(\lambda)= a$, and $d(k,n)=a^k$ from  formula \eqref{dnu}. Hence, the self-duality functions must be of the form
\beq \nonumber
D(\xi,\eta) &=& \prod_{x \in V}d(\xi_x,\eta_x)\ =\ a^{|\xi|}\ ,
\eeq
i.e.\ depending only on the total number of dual particles (and not on the configuration $\eta$). Hence, the duality relation in that case reduces to the trivial relation, for all $t \geq 0$ and $\xi \in \hat \Omega$,
\[
\E_\xi a^{|\xi_t|}= a^{|\xi|}\ ,\quad a \in \R\ ,
\]
which is just conservation of the number of particles in the dual process. No other self-duality relation with factorized self-duality functions
can exist.

\subsection{Interacting diffusions and orthogonal polynomial duality functions}\label{section Laplace}
As shown in Theorem \ref{finitethm}, relation \eqref{reldual5} still holds whenever the discrete  right-variables $n \in \N$ are replaced by continuous variables $z \in \R_+$ and sums by integrals. With this observation in mind, we provide a second general method to \emph{characterize all factorized duality functions} between the continuous process $\BEP(\alpha)$ and its discrete dual $\SIP(\alpha)$.

More precisely, if $d(k,z)$ is a single-site duality function with property \eqref{Da}  between  $\BEP(\alpha)$
and  $\SIP(\alpha)$, and $\nu_\lambda$ is the stationary product measure marginal for $\BEP(\alpha)$ as in \eqref{gamma distribution}, then, from the analogue of relation \eqref{reldual5} for $k=1$, namely
\beq \label{reldual1c}
\int_{\R_+} d(1,z) z^{\alpha-1} e^{-\lambda z}\frac{\lambda^\alpha}{\Gamma(\alpha)}  dz &=& \theta(\lambda)\ ,
\eeq
we necessarily have by Theorem \ref{finitethm} that
\beq \label{reldual2c}
\int d(k,z) \frac{z^{\alpha-1}}{\Gamma(\alpha)} e^{-\lambda z} dz &=& \theta(\lambda)^k  \lambda^{-\alpha}\  .
\eeq
As a consequence, the function $d(k,z) \frac{z^{\alpha-1}}{\Gamma(\alpha)}$ is the inverse Laplace transform of $\theta(\lambda)^k \lambda^{-\alpha}$.
Given the first single-site duality function $d(1,z)$ in \eqref{ghgh}, from \eqref{reldual1c} we obtain
\beq \nonumber
\theta(\lambda) &=& \int (a+bz) z^{\alpha-1} e^{-\lambda z} \frac{\lambda^\alpha}{\Gamma(\alpha)}  dz = \left(a \lambda + b \alpha \right) \lambda^{-1}\ .
\eeq
As a consequence, the r.h.s.\ in \eqref{reldual2c} becomes
\beq \label{hjhj}
\theta(\lambda)^k \lambda^{-\alpha} &=& \left(a \lambda + b \alpha \right)^k \lambda^{ - \alpha-k}\ ,
\eeq
and there exist explicit expressions for the inverse Laplace transform of this function. We  split the computation in  two cases.
In case $a=0$, since the inverse Laplace transform of $\lambda^{-\alpha-k}$ is $\frac{z^{\alpha+k-1}}{\Gamma(\alpha+k)}$, we  immediately obtain
\beq \nonumber
d(k,z) &=& (b \alpha)^k \frac{z^k \Gamma(\alpha)}{\Gamma(\alpha+k)}\ ,
\eeq
i.e.\ the \textquotedblleft classical" single-site duality function as in Section \ref{section setting diffusion}, up to set $b=\frac{1}{\alpha}$.
In case $a \neq  0$, the inverse Laplace transform of \eqref{hjhj} is more elaborated:
\beq \nonumber
a^k\ \frac{z^{\alpha-1}}{\Gamma(\alpha)}\ \pFq{1}{1}{-k}{\alpha}{-\frac{b \alpha}{a} z}\ .
\eeq
As the above expression must equal $d(k,z) \frac{z^{\alpha-1}}{\Gamma(\alpha)}$, it follows that
\beq \nonumber
d(k,z) &=& a^k\ \pFq{1}{1}{-k}{\alpha}{-\frac{b \alpha}{a} z}\ .
\eeq
As a final consideration, we  note that for the choice $a \cdot b < 0$,
\beq
d(k,z) &=& a^k\ \frac{k! \Gamma(\alpha)}{\Gamma(\alpha+k)}\ L_k(z; \alpha-1, \mu)\ ,
\eeq
where $\mu= -\frac{b \alpha}{a}$ here and $\{L_k(z; \alpha-1, \mu),\ k \in \N  \}$ are the generalized Laguerre polynomials - orthogonal polynomials w.r.t.\ to the Gamma distribution of scale parameter $\frac1\mu$ and shape parameter $\alpha$ as defined in \cite{Nikiforov}.

\section{Intertwining and generating functions}\label{section generating}

In this section, we introduce the generating function method, which allows to go from a self-duality of a
discrete process towards duality between a discrete and continuous process,  and further towards a self-duality of a
continuous process, and back. This then allows e.g.\ to simplify the proof of a discrete self-duality  by lifting it to
a continuous self-duality, which is usually easier to verify. The key of this method is an intertwining  between discrete multiplication and derivation operators and
their continuous analogues, via an appropriate generating function.

To reduce issues of well-definition of operators and their domains, in this section we restrict our discussion to the case of  finite vertex set $V$.

\subsection{Introductory example}

To make the method clear, let us start with a simple example of independent random walkers on a single edge.
The generator is
\beq \nonumber
L f(n_1,n_2)&=& n_1 (f(n_1-1, n_2+1)- f(n_1,n_2)) + n_2  (f(n_1+1, n_2-1)- f(n_1,n_2))\ ,
\eeq
with $n_1, n_2 \in \N$.
Define now the (exponential) generating function
\beq \nonumber
\caG f(z_1,z_2)&=& \sum_{n_1,n_2=0}^\infty f(n_1,n_2) \frac{z_1^{n_1}}{n_1!} \frac{z_2^{n_2}}{n_2!}\ ,\quad z_1, z_2 \in \R_+\ .
\eeq
Then it is easy to see that
\beq \nonumber
\caL \caG &=& \caG L\ ,
\eeq
where
\beq \nonumber
\caL &=& -(z_1-z_2)(\partial_{z_1}-\partial_{z_2})\ .
\eeq
Now assume that we have a self-duality function for the particle system, i.e.\
\beq \nonumber
L_\eft D&=& L_\ight D\ ,
\eeq
then we have
\beq \nonumber
L_\eft \caD &=& \caL_\ight \caD\ ,
\eeq
where
\beq \nonumber
\caD(k_1, k_2; z_1, z_2)\ =\ \caG_{\ight} D(k_1, k_2; z_1, z_2)\ =\ \sum_{n_1, n_2=0}^\infty D(k_1,k_2; n_1,n_2) \frac{z_1^{n_1}}{n_1!} \frac{z_2^{n_2}}{n_2!}\ .
\eeq
In words, a self-duality function of $L$ is ``lifted'' to a duality function between the independent random walk generator $L$
and its continuous counterpart  $\caL$ by applying the generating function to the $n$-variables.
Conversely, given a duality function between the independent random walk generator $L$
and its continuous counterpart  $\caL$, its Taylor coefficients provide a self-duality function of $L$.

We can then also take the generating function w.r.t.\ the $k$-variables in the function  $\caD$ to produce
a self-duality function for $\caL$, i.e.\ defining
\beq \nonumber
\mathscr D(v_1, v_2; z_1, z_2)\ =\ \caG_\eft \caD(v_1, v_2; z_1, z_2)\ =\ \sum_{k_1, k_2=0}^\infty \caD(k_1, k_2; z_1, z_2) \frac{v_1^{k_1}}{k_1!} \frac{v_2^{k_2}}{ k_2!}\ ,
\eeq
we have
\beq \nonumber
\caL_\eft \mathscr D &=& \caL_\ight \mathscr D\ .
\eeq
For the classical self-duality function $D(k_1,k_2; n_1,n_2)=  \frac{n_1!}{(n_1-k_1)!}\frac{n_2!}{(n_2-k_2)!}\1\{k_1\leq n_1\}\1\{k_2\leq n_2\}$,
we find that
\beq \nonumber
\mathscr D(v_1, v_2; z_1, z_2)\ =\ e^{z_1+z_2} e^{v_1 z_1+ v_2 z_2}\ .
\eeq
Beside the factor $e^{z_1+z_2}$ which depends only on the conserved quantity $z_1+z_2$, to check the self-duality relation for $\caL$ w.r.t.\ the function $e^{v_1 z_1+ v_2 z_2}$ is rather straightforward, the computation involving only derivatives of exponentials. By looking at the Taylor coefficients w.r.t.\ both $v$ and $z$-variables of this self-duality relation, we obtain the self-duality relation for $L$ w.r.t.\ $D$ where we started from.

In conclusion, all these duality relations turn out to be equivalent, and the proof of self-duality for particle systems requiring rather intricate combinatorial arguments (cf.\ e.g.\ \cite{demasi}) is superfluous once the more direct self-duality for diffusion systems is checked.

\subsection{Intertwining and duality}

The notion of   \emph{intertwining} between stochastic processes was originally introduced by Yor in \cite{Yor} in the context of Markov chains and later pursued  in \cite{DF} and \cite{F} as an abstract framework, in discrete-time and continuous-time respectively, for the problem of Markov functionals, i.e.\ finding sufficient and necessary conditions under which a random function of a Markov chain is again Markovian.

For later purposes, we adopt a  rather general  definition of intertwining, in which $\{\eta(t),\ t \geq 0 \}$ and $\{\zeta(t),\ t \geq 0 \}$ are continuous-time stochastic processes on the Polish spaces $\Omega$ and $\Omega'$, respectively, whose expectations read $\E$, $\E'$ resp.,  and  $\caM( \Omega)$ denotes the space of signed measures on $\Omega$. We say that $\{\zeta(t),\ t \geq 0 \}$ is \emph{intertwined on top} of $\{\eta(t),\ t \geq 0 \}$ if there exists a mapping   $\varLambda: \Omega' \to \caM(\Omega)$   such that, for all $t \geq 0$, $\zeta \in \Omega'$ and $f : \Omega \to \R$,
\beq \label{intertwin}
\E'_\zeta \int_{\Omega'} f(\eta)\ \varLambda(\zeta(t))(d\eta) &=&\int_{\Omega} \E_{\eta} f(\eta(t))\ \varLambda(\zeta)(d\eta)\ .
\eeq
Working at the abstract level of \emph{semigroups}, we say that $\{\caS(t),\ t \geq 0 \}$
on a space of functions $f:\Omega'\to\R$ denoted by $\caF(\Omega')$, is \emph{intertwined on top} of  $\{S(t),\ t \geq 0 \}$, a semigroup on a space
of functions $f:\Omega\to\R$ denoted by $\caF(\Omega)$, with \emph{intertwiner} $\varLambda$ if $\varLambda$  is a linear operator
from $\caF(\Omega)$ into $\caF(\Omega')$ and if, for all $t \geq 0$ and $f : \Omega \to \R$,
\beq \nonumber
  \caS(t) \varLambda f&=& \varLambda S(t)  f\ .
\eeq
Similarly, \emph{operators} $\caL$ with domain $\caD(\caL)$  and $L$  with domain $\caD(L)$ are \emph{intertwined} with intertwiner $\varLambda$ if, for all $f \in \caD(L)$, $\varLambda f\in \caD(\caL)$ and
\beq \label{opintertwin}
\caL \varLambda f &=& \varLambda L  f\ .
\eeq
Notice that with a slight abuse of notation we used the same symbol $\varLambda$ for an abstract intertwining operator as for
the intertwining mapping. In other words, in case the intertwining mapping as in \eqref{intertwin} is given by $\tilde \varLambda$, then the corresponding operator is
\beq \nonumber
\varLambda f(\zeta) &=& \int f(\eta)\ {\tilde \varLambda(\zeta)}(d\eta)\ .
\eeq

An intertwining mapping $\varLambda$ has a  probabilistic interpretation if it takes values in the subset of probability measures on $\Omega$. Indeed, in \eqref{intertwin} the process $\{\zeta(t),\ t \geq 0 \}$ may be viewed as an added structure on top of $\{\eta(t),\ t \geq 0 \}$ or, alternatively, the process $\{\eta(t),\ t \geq 0 \}$ as a random functional of $\{\zeta(t),\ t \geq 0\}$, in which $\varLambda$ provides this link.
\br
The connection with duality introduced in Section \ref{section setting duality} becomes transparant when $\Omega$, $\hat \Omega$ and $\Omega'$ are finite sets and the operators and functions $L$, $\hat L$ $\caL$, $D$ and $\varLambda$ in \eqref{opdual} and \eqref{opintertwin} are represented in terms of matrices. There, relations \eqref{opdual} and \eqref{opintertwin}, once rewritten in matrix notation as
\beq \label{matrixdual}
\hat L D &=& D L^\dagger\ ,
\eeq
where $L^\dagger$ denotes the transpose of $L$,
and
\beq \label{matrixintertwin}
\caL \varLambda &=& \varLambda L\ ,
\eeq
differ essentially only in the terms $L^\dagger$ versus $L$ in the r.h.s.\ of both identities.
 The presence or absence of transposition can be interpreted as a \emph{forward}-versus-\emph{backward} evolution against a \emph{forward}-versus-\emph{forward} evolution. More precisely, if $L$, $\hat L$ and $\caL$ are generators of Markov processes $\{\eta(t), t \geq 0\}$, $\{\xi(t),\ t \geq 0 \}$ and $\{\zeta(t),\ t \geq 0 \}$, respectively, then \eqref{matrixdual} and \eqref{matrixintertwin} relate the evolution of $\{\eta(t),\ t \geq 0\}$ to that of $\{\xi(t),\ t \geq 0 \}$, resp.\ $\{\zeta(t),\ t \geq 0 \}$; however, while in \eqref{matrixintertwin} the processes run both along the same direction in time, in \eqref{matrixdual} the processes run along opposite time directions.
\er
Intertwiners as $\varLambda$ in \eqref{matrixintertwin} may be also interpreted as  natural generalizations of {\em symmetries of generators}, indeed \eqref{matrixintertwin} with $\caL = L$ just means that $\Lambda$ commutes with $L$, which is the definition of a symmetry of $L$. As outlined in \cite[Theorem 2.6]{gkrv}, the knowledge of symmetries of a generator and dualities of this generator leads to the construction of new dualities. The following theorem presents the analogue procedure in presence of intertwiners: a duality and an intertwining lead to a new duality.

\begin{theorem}\label{theorem intertwin}
	Let $L$, $\hat L$ and $\caL$ be operators on real-valued functions on $\Omega$, $\hat \Omega$ and $\Omega'$, respectively. Suppose that there exists an intertwiner $\varLambda$ such that for all $f  \in \caD(L)$, $\varLambda f \in \caD(\caL)$,
	\beq \nonumber
	\caL \varLambda  f &=& \varLambda L  f\ ,
	\eeq
	and a duality function $D: \hat \Omega \times \Omega \to \R$ for $\hat L$ and $L$, namely $D(\xi,\cdot) \in \caD(L)$ for all $\xi \in \hat \Omega$, $D(\cdot,\eta) \in \caD(\hat L)$ for all $\eta \in \Omega$  and
	\beq \nonumber
	\hat L_\eft D &=& L_\ight D\ .
	\eeq
	Then, if $\varLambda_\ight D(\xi,\cdot) \in \caD(\caL)$ for all $\xi \in \hat \Omega$ and $\varLambda_\ight D(\cdot,\zeta) \in \caD(\hat L)$ for all $\zeta \in \Omega'$, $\varLambda_\ight D$ is a duality function for $\hat L$ and $\caL$, i.e.
	\beq \nonumber
	\hat L_\eft \varLambda_\ight D &=& \caL_\ight \varLambda_\ight D\ .
	\eeq
\end{theorem}
\begin{proof}
	\beq \nonumber
	\hat L_\eft \varLambda_\ight D\ =\ \varLambda_\ight \hat L_\eft D\ =\ \varLambda_\ight L_{\ight} D\ =\ \caL_\ight \varLambda_\ight D\ .
	\eeq
Here in the first equality we used that left and right actions commute, in the second equality we used the assumed duality of $\hat{L}$ and $L$,
and in the third equality we used the assumed intertwining.
\end{proof}

\subsection{Intertwining between continuum and discrete processes}
In this section we prove the existence of an intertwining relation between the interacting diffusion processes presented in Section \ref{section setting diffusion} and the particle systems of Section \ref{section setting particle}. This intertwining relation provides a second connection, besides the scaling limit procedure (cf.\ Section \ref{section setting diffusion}), between continuum and discrete processes, which proves to be better suited for the goal of establishing duality relations among these processes. Indeed, the characterization of all possible factorized self-dualities for particle systems obtained in Section \ref{section general construction} and the intertwining relation below, via the application of Theorem \ref{theorem intertwin}, produces a characterization of all possible dualities, resp.\ self-dualities, between the discrete and the continuum processes, resp.\ of the continuum process.

\

In the following proposition, we prove the intertwining relation for operators $L^{\sigma,\beta}$ and $\caL^{\sigma,\beta}$ defined, respectively, on functions $f : E^V \to \R$, $E= \N$, as
\beq \label{Lsigmabeta}
L^{\sigma,\beta} f(\eta) &=& \sum_{x,y \in V} p(x,y)\ L^{\sigma,\beta}_{x,y} f(\eta)\ ,\quad \eta \in E^V\ ,
\eeq
where
\beq \nonumber
L^{\sigma,\beta}_{x,y}f(\eta) &=& \eta_x(\beta+\sigma\eta_y) (f(\eta^{x,y})-f(\eta)) + \eta_y (\beta+\sigma \eta_x) (f(\eta^{y,x})-f(\eta))\ ,
\eeq
and, on real analytic functions $f : \R^V \to \R$, as
\beq \label{Lsigmabetadiffusion}
\caL^{\sigma,\beta} f(z) &=& \sum_{x,y \in V} p(x,y)\ \caL^{\sigma,\beta}_{x,y}f(z)\ ,\quad z \in \R^V\ ,
\eeq
where
\beq \nonumber
\caL^{\sigma,\beta}_{x,y}f(z) &=& (-\beta (z_x-z_y)(\partial_x-\partial_y)  +\sigma z_x z_y (\partial_x-\partial_y)^2 )f(z)\ .
\eeq 
Note that $L^{\sigma,\beta}$ in \eqref{Lsigmabeta} is a special instance of the generator $L^{u,v}$ in \eqref{gzrpgen} with conditions \eqref{didi}, while $\caL^{\sigma,\beta}$  above,  matches on a common sub-domain, for particular choices of the parameters $\sigma$ and $\beta$, those in \eqref{generators diffusions}.

\begin{proposition}
	Let $G$ be the Poisson probability kernel defined as the   operator that maps functions $f : \N \to \R$ into functions $G f : \R \to \R$ as
	\beq \label{G single-site intertwiner}
	G f(z) &=& \sum_{n =0}^\infty f(n) \frac{z^n}{n!} e^{-z}\ ,\quad z \in \R\ .
	\eeq
	Then, whenever $G f : \R \to \R$ is a real analytic function, if $G^\otimes= \otimes_{x \in V} G^{(x)}$ denotes the tensorized operator mapping functions $f : \N^V \to \R$ into functions $f : \R^V \to \R$ accordingly, $\caL^{\sigma,\beta}$ and $L^{\sigma,\beta}$ are intertwined with intertwiner $G^\otimes$, namely
	\beq \label{intertwinL}
	\caL^{\sigma,\beta} G^\otimes f(z) &=& G^\otimes L^{\sigma,\beta} f(z)\ ,\quad z \in \mathbb \R^V\ .
	\eeq
\end{proposition}
\begin{proof}Let us introduce the non-normalized operator
\beq \nonumber
\bar G f(z) &=& \sum_{n=0}^\infty f(n) \frac{z^n}{n!}\ ,\quad z \in \R\ ,
\eeq
and the associated tensorized operator $\bar G^\otimes =\otimes_{x \in V} \bar G^{(x)}$.
	Due to the factorized structure of $L^{\sigma,\beta}$, $\caL^{\sigma,\beta}$ and $\bar G^\otimes$, the proof of the intertwining relation \eqref{intertwinL} with  $\bar G^\otimes$ as an intertwiner reduces to consider and combine the following relations:
	\beq \nonumber
	\sum_{n=0}^\infty n f(n-1)
	\frac{z^n}{n!} &=& z \bar Gf(z)\\
	\nonumber
	\sum_{n=0}^\infty f(n+1) \frac{z^n}{n!}  &=&  \frac{d}{dz} \bar Gf(z)\\
	\nonumber
	\sum_{n=0}^\infty n  f(n) \frac{z^n}{n!} &=& z \frac{d}{dz} \bar Gf(z)\\
	\nonumber
	\sum_{n=0}^\infty n f(n+1) \frac{z^n}{n!} &=& z \frac{d^2}{dz^2} \bar Gf(z)\ .
	\eeq
	As a first consequence, we have
	\beq \nonumber
	\caL^{\sigma,\beta} \bar G^\otimes  f(z) &=&  \bar G^\otimes  L^{\sigma,\beta} f(z)\ ,\quad z \in \R^V\ .
	\eeq
	We obtain \eqref{intertwinL} by observing  that ($|z|=\sum_{x \in V} z_x$)
	\beq \nonumber
	G^\otimes f(z) &=& e^{-|z|}\ \bar G^\otimes f(z)\ ,\quad z \in \R^V\  ,
	\eeq
	and that, for $g(z) = \bar g(z)\cdot e^{-|z|}$,
	\beq \nonumber
	\caL^{\sigma,\beta} g(z) &=& e^{-|z|}\ \caL^{\sigma,\beta} \bar g(z)\ ,\quad z \in \R^V\ .
	\eeq
\end{proof}
We note that the intertwiner $G^\otimes$ has a nice probabilistic interpretation: from an \textquotedblleft energy\textquotedblright\ configuration $z \in \R_+^V$, the associated particle configurations are generated by placing a number of particles on each site $x \in V$, independently, and distributed according to a Poisson random variable with intensity $z_x$.

In the remaining part of this section, under some reasonable regularity assumptions, we are able to invert the intertwining relation \eqref{intertwinL}, namely to find an operator $H^\otimes=\otimes_{x \in V} H^{(x)}$ that intertwines $L^{\sigma,\beta}$ and $\caL^{\sigma,\beta}$, in this order. The natural candidate for $H$ is the \textquotedblleft inverse operator\textquotedblright\ of $G$, whenever this is well-defined. In general, this \textquotedblleft inverse intertwiner\textquotedblright\ lacks any probabilistic interpretation, but indeed establishes a second intertwining relation useful in the next section.

\begin{proposition}\label{proposition inverse intertwining}
	Let $H$ be the differential operator mapping  real analytic  functions $g : \R \to \R$ into functions $H g : \N \to \R$ as
	\beq \label{H single-site inverse intertwiner}
	H g(n) &=& \left(\left[\frac{d^n}{d z^n} \right]_{z=0} e^z g(z) \right)\ ,\quad n \in \N\ .
	\eeq
	Then $H$ is the inverse operator of $G$, namely, for all $f : \N \to \R$ such that $G f : \R \to \R$ is real analytic, we have
	\beq \nonumber
	G H g (z)\ =\ g(z)\ ,\quad z \in \R\ ,\quad H G f(n)\ =\ f(n)\ ,\quad n \in \N\ .
	\eeq
	Moreover, the tensorized operator $H^\otimes=\otimes_{x \in V} H^{(x)}$ is an intertwiner for $L^{\sigma, \beta}$ and $\caL^{\sigma,\beta}$, i.e.\ for all real analytic $g : \R^V \to \R$,
	\beq \label{intertwinLinv}
	L^{\sigma,\beta} H^\otimes g(n) &=& H^\otimes \caL^{\sigma,\beta} g(n)\ ,\quad n \in \N\ .
	\eeq
\end{proposition}
Before giving the proof, we need the following lemma.

\begin{lemma}\label{lemma symmetry}
Let $A$ be the operator acting on functions $f : \N \to \R$ defined as
 \beq \nonumber
A f(n) &=& \sum_{k=0}^n \binom{n}{k} f(k)\ ,\quad n \in \N\ .
\eeq
Then, the tensorized operator $A^\otimes = \otimes_{x \in V} A^{(x)}$ is a symmetry for the generator $L^{\sigma,\beta}$, i.e.\ for all $f\ : \N^V \to \R$
\beq \label{symmetry}
A^\otimes L^{\sigma, \beta} f(n) &=& L^{\sigma,\beta} A^\otimes f(n)\ ,\quad n \in \N^V\ .
\eeq
\end{lemma}
\begin{proof}
	Instead of going through tedious computations, we  exploit the fact that the operator $A^\otimes$ has the form
	\beq \nonumber
	A^\otimes &=& \otimes_{x \in V}\ A^{(x)}\ =\ \otimes_{x \in V}\ e^{J^{(x)}}\ =\ \otimes_{x \in V}\ \sum_{k=0}^\infty \frac{(J^{(x)})^k}{k!}\ ,
	\eeq
	where $J^{(x)}$ is an operator defined for functions $f : \N^V \to \R$ which acts only on the $x$-th variable as
	\beq \nonumber
	J^{(x)} f(n) &=& n_x\  f(n-\delta_x)\ ,\quad n \in \N^V\ .
	\eeq
	Since all these operators $\{J^{(x)},\ x \in V \}$ commute over the sites we have
	\beq \nonumber
	A^\otimes &=& \otimes_{x \in V} e^{J^{(x)}}\ =\ e^{\sum_{x \in V} J^{(x)}}\ .
	\eeq
	We conclude the proof by noting that the operator
	\beq \nonumber
	\sum_{x \in V} J^{(x)}
	\eeq
	is a symmetry for the generator $L^{\sigma, \beta}$, cf.\ e.g.\  \cite{gkrv}, \cite{Transport}.
\end{proof}

\begin{proof}[Proposition \ref{proposition inverse intertwining}]
	First we compute the following key relations:
	\beq \nonumber
	\left(\left[\frac{d^n}{dz^n} \right]_{z=0} g'(z)\right) &=& \left(\left[\frac{d^{n+1}}{dz^{n+1}} \right]_{z=0} g(z)\right)\\
	\nonumber
	\left(\left[\frac{d^n}{dz^n} \right]_{z=0} z g(z)\right) &=& n \left(\left[\frac{d^{n-1}}{dz^{n-1}} \right]_{z=0} g(z)\right)\\
	\nonumber
	\left(\left[\frac{d^n}{dz^n} \right]_{z=0} z g'(z)  \right) &=& n \left( \left[\frac{d^n}{dz^n} \right]_{z=0} g(z)\right)\\
	\nonumber
	\left(\left[\frac{d^n}{dz^n} \right]_{z=0} z g''(z) \right) &=& n \left(\left[\frac{d^{n+1}}{dz^{n+1}} \right]_{z=0} g(z) \right)\ .
	\eeq
	Hence, if we introduce the operator
	\beq \nonumber
	\bar H g(n) &=& \left(\left[\frac{d^n}{d z^n} \right]_{z=0} g(z) \right)\ ,\quad n \in \N\ ,
	\eeq
	and the associated tensorized operator $\bar H^\otimes = \otimes_{x \in V} \bar H^{(x)}$, we obtain
	\beq \label{Hbar intertwining}
	L^{\sigma,\beta} \bar H^\otimes g(n) &=& \bar H^\otimes \caL^{\sigma,\beta} g(n)\ ,\quad n \in \N\ .	
	\eeq
	Now, by using Lemma \ref{lemma symmetry} and  noting that
	\beq \nonumber
	H g(n) &=& \sum_{k=0}^n \binom{n}{k} \bar H g(k)\ =\  A \bar H g(n)\ ,\quad n \in \N\ ,
	\eeq
	and, by the mixed property of the tensor product, 
	\beq \label{HvsHbar}
	H^\otimes g(n) &=& (A \bar H)^\otimes g(n)\ =\ A^\otimes \bar H^\otimes g(n)\ ,\quad n \in \N^V\ ,
	\eeq
	we get \eqref{intertwinLinv}  by applying first \eqref{HvsHbar}, then \eqref{symmetry} and finally \eqref{Hbar intertwining}:
	\beq \nonumber
	L^{\sigma,\beta} H^\otimes g\ =\ L^{\sigma,\beta} A^\otimes \bar H^\otimes g\ =\ A^\otimes L^{\sigma,\beta} \bar H^\otimes g\ =\ A^\otimes \bar H^\otimes \caL^{\sigma,\beta} g\ =\ H^\otimes \caL^{\sigma,\beta} g\ .
	\eeq
	\end{proof}

\subsection{Generating functions and duality}
As anticipated in the previous section, from the intertwining relation \eqref{intertwinL} and the functions obtained in Section \ref{section particle systems duality}, in what follows we find explicitly new duality relations.

Due to the factorized form \eqref{factoo} of the self-duality functions with single-site functions \eqref{single-site irw} and \eqref{dknhyper} and the tensor form of the intertwiner $G^\otimes$ in \eqref{intertwinL}, the new functions inherit the same factorized form. Moreover, from the definition of $G$ in \eqref{G single-site intertwiner}, the whole computation reduces to determine \emph{(exponential) generating functions} of \eqref{single-site irw} and \eqref{dknhyper}.  To this purpose, some identities for hypergeometric functions are available, cf.\ e.g.\  the tables in \cite[Chapter 9]{koekoek}. Moreover, all  generating functions obtained satisfy the requirements of analyticity for suitable choices of the parameters $\sigma$, $\beta$, $a$ and $b$ (cf.\ \cite{koekoek}), hence all operations below make sense.

However, just as the functions found in Section \ref{section particle systems duality}, the functions here obtained will only be \textquotedblleft candidate\textquotedblright\ (self-)duality functions, since no duality relation as in \eqref{dualrel} has been proved, yet. By using the \textquotedblleft inverse\textquotedblright\  intertwining \eqref{intertwinLinv}, all these \textquotedblleft possible\textquotedblright\ dualities turn out to be equivalent, i.e.\ one implies all the others. Thus, in Proposition \ref{proposition eee} below,  we choose to prove directly the self-duality relation for the continuum process, more immediate to verify due to the simpler form of the self-duality functions. Indeed, while the single-site self-duality functions for the $\SIP(\alpha)$ process, for instance,  have the generic form of an hypergeometric function
\beq \nonumber
\pFq{2}{1}{-k,-n}{\alpha}{\frac{b}{a}\alpha}\ ,\quad k, n \in \N\ ,
\eeq
the single-site duality functions between discrete and continuum processes involve in their expressions hypergeometric functions
\beq \nonumber
\pFq{1}{1}{-k}{\alpha}{-\frac{b}{a}\alpha z}\ ,\quad k \in \N\ ,\ z \in \R_+\ ,
\eeq
and those for the self-duality of continuum processes are even simpler, namely
\beq \nonumber
\pFq{0}{1}{-}{\alpha}{\frac{b}{a}\alpha v z}\ ,\quad v, z \in \R_+\ ,
\eeq
as the number of arguments of the hypergeometric function drops.

\medskip

The tables below schematically report all single-site (self)-duality functions for the operators $L^{\sigma,\beta}$ in \eqref{Lsigmabeta} and  $\caL^{\sigma,\beta}$ in \eqref{Lsigmabetadiffusion}.  Recall that the parameters $a$, $b \in \R$ in \eqref{didi} are properly chosen (cf.\ Section \ref{section particle systems duality}).

\

\begin{itemize*}
	\item[(I)] \textbf{Independent random walkers ($\IRW$)\ , $\sigma=0\ $, $\beta=1\ $.}
\end{itemize*}

\medskip

\begin{center}
	\begin{tabular}{c||c|c|c}
		\hline
		\hline
		&&&\\
		Classical polynomials & $\frac{n!}{(n-k)!} b^k\1\{k \leq n\}$ &$ \left(bz\right)^k$&$e^{- v} e^{ b v z}$\\
		&&&\\
		\hline
		\hline
		&&&\\
		Orthogonal polynomials &$a^k C_k(n;-\frac{a}{b})$& $  \left(a + b z \right)^k$&$e^{(a-1)v}e^{b v z}$\\
		&&&\\
		\hline
		\hline
		&&&\\
		Cheap duality functions &$e^{\lambda} \frac{k!}{\lambda^k}\1\{k=n \}$&$e^{\lambda-z}\left(\frac{z}{\lambda} \right)^k$&$e^{\lambda-z-v}e^{ \frac{v z}{\lambda}}$\\
		&&&\\
		\hline \hline
	\end{tabular}
\end{center}

\bigskip
\medskip

\begin{itemize*}
	\item[(II)] \textbf{Symmetric inclusion process ($\SIP(\alpha)$)\ , $\sigma=1\ $, $\beta=\alpha> 0\ $.}
\end{itemize*}

\medskip

\begin{center}
	\begin{tabular}{c||c|c|c}
		\hline
		\hline
		&&&\\
		Cl. & $\frac{n!}{(n-k)!} \frac{\Gamma\left(\alpha \right)}{\Gamma\left(\alpha  + k\right)} \left(b \alpha \right)^k \1\{k \leq n\}$ &$ \frac{\Gamma(\alpha)}{\Gamma(\alpha+k)} \left(b \alpha z \right)^k$&$ e^{-v} \pFq{0}{1}{-}{\alpha}{b \alpha v z}$\\
		&&&\\
		\hline
		\hline
		&&&\\
		Or. &$a^k \frac{\Gamma\left(\alpha \right)}{\Gamma\left(\alpha + k \right)} M_k(n;  \frac{a}{a-b\alpha}, \alpha)$& $ a^k \frac{k! \Gamma(\alpha)}{\Gamma(\alpha+k)} L_k(z; \alpha-1, -\frac{b}{a}\alpha)$&$e^{(a-1) v}  \pFq{0}{1}{-}{\alpha}{b\alpha v z}$\\
		&&&\\
		\hline
		\hline
		&&&\\
		Ch. &$(1-\lambda)^\alpha\frac{k!}{\lambda^k} \frac{\Gamma(\alpha)}{\Gamma(\alpha+k)} \1\{k=n \}$&$(1-\lambda)^\alpha e^{-z} \left(\frac{z}{\lambda}\right)^k \frac{\Gamma(\alpha)}{\Gamma(\alpha+k)}$&$ (1-\lambda)^\alpha e^{-z-v}\pFq{0}{1}{-}{\alpha}{\frac{v z}{\lambda}}$\\
		&&&\\
		\hline \hline
	\end{tabular}
\end{center}

\bigskip
\medskip

\begin{itemize*}
	\item[(III)] \textbf{Symmetric exclusion process ($\SEP(\gamma)$)\ ,\ $\sigma=-1\ $,\ $\beta=\gamma \in \N\ $.}
\end{itemize*}

\medskip

\begin{center}
	\begin{tabular}{c||c|c|c}
		\hline
		\hline
		&&&\\
		Cl. & $\tfrac{(\gamma-k)!}{\gamma!}\tfrac{n!}{(n-k)!}(b\gamma)^k \1\{k \leq n\}$ &$ \tfrac{(\gamma-k)!}{\gamma!} \left(b \gamma z \right)^k$&$ e^{-v} \pFq{0}{1}{-}{-\gamma}{-b \gamma v z}$\\
		&&&\\
		\hline
		\hline
		&&&\\
		Or. &$a^k K_k(n; -\tfrac{a}{b\gamma}, \seppar ) \left( \frac{b\gamma}{a} \right)^k \frac{1}{\binom{\seppar}{k}}$& $ a^k \pFq{1}{1}{-k}{-\gamma}{\frac{b}{a} \gamma z}$&$e^{(a-1) v}  \pFq{0}{1}{-}{-\gamma}{-b\gamma v z}$\\
		&&&\\
		\hline
		\hline
		&&&\\
		Ch. &$(1+\lambda)^{-\gamma}\frac{k!}{\lambda^k} \tfrac{\gamma!}{(\gamma-k)!} \1\{k=n\}$&$(1+\lambda)^{-\gamma}e^{-z}\left(\frac{z}{\lambda}\right)^k \tfrac{\gamma!}{(\gamma-k)!}$&$e^{-z-v} \left( \frac{\lambda+v z}{\lambda(1+\lambda)}\right)^{\gamma}$\\
		&&&\\
		\hline \hline
	\end{tabular}
\end{center}

\medskip

 More in detail, on the left-most column we place the single-site self-duality functions $d(k,n)$ for the particle systems of Section \ref{section setting particle}: while the top-left functions are those already appearing in \cite{gkrv}, \cite{Transport}, cf.\  also Section \ref{section setting particle} and \eqref{dknclassical} - and, thus, for this reason denoted here as the \textquotedblleft classical\textquotedblright\ ones - the second-to-the-top functions are those derived in Section \ref{section particle systems duality} in  \eqref{single-site irw}--\eqref{dknhyper} and being related to suitable families of orthogonal polynomials. While these two classes of single-site self-duality functions satisfy condition \eqref{Da} (they are the only ones doing so), the bottom-left single-site self-duality functions correspond to the \textquotedblleft cheap\textquotedblright\ self-duality (cf.\ end of Section \ref{section setting lattice}), namely the detailed-balance condition w.r.t.\ the measures $\{\otimes_{x \in V} \nu_\lambda,\ \lambda > 0 \}$ with marginals \eqref{marginals}.

On the mid-column, we find the single-site duality functions between the difference operators $L^{\sigma,\beta}$ and the differential operators $\caL^{\sigma,\beta}$, obtained from their left-neighbors by a direct application of the operator $G$ in \eqref{G single-site intertwiner} on the $n$-variables. The new functions will depend hence on the two variables $k \in \N$ and $z \in \R_+$.

 A second application w.r.t.\ the $k$-variables of the same operator $G$ on the functions just obtained gives us back the right-most column, functions depending now on variables $v, z \in \R_+$. These functions represent the single-site self-duality functions for the differential operator $\caL^{\sigma,\beta}$. As an immediate consequence of Proposition \ref{proposition inverse intertwining}, we could also proceed from right to left by applying the inverse intertwiner $H$ in \eqref{H single-site inverse intertwiner}.

Note that the single-site self-duality functions for $\caL^{\sigma,\beta}$ on the right-most columns, though they have been derived from different discrete analogues, i.e.\ classical, orthogonal and cheap single-site functions, within the same table they differ only for a factor which depends only on the conserved quantities $|z|=\sum_{x \in V}z_x$ and $|v|=\sum_{x \in V}v_x$. Henceforth, when proving the self-duality relation, this extra-factor does not play any role and it is enough to check that the functions
\beq \label{dvz}
d(v,z)\ =\ e^{c v z}\ ,\quad d(v,z)\ =\ \pFq{0}{1}{-}{\alpha}{c v z}\ ,\quad d(v,z)\ =\ \pFq{0}{1}{-}{-\gamma}{c v z}\ ,\quad v, z \in \R_+\ ,
\eeq
for constants $c \in \R$, are single-site self-duality functions for the operators $\caL^{0,1}$, $\caL^{1,\alpha}$ and $\caL^{-1,\gamma}$, respectively. This final computation is the content of the next proposition.

\begin{proposition}\label{proposition eee}
	For any constant $c \in \R$, the functions $d(v,z)$ in \eqref{dvz} are single-site self-duality functions for the differential operators $\caL^{0,1}$, $\caL^{1,\alpha}$ and $\caL^{-1,\gamma}$, respectively.
\end{proposition}
\begin{proof}	
	To prove that
	\beq \nonumber
	d(v,z) &=& e^{c v z}\ ,\quad v, z \in \R_+\ ,
	\eeq
	is a single-site self-duality function for the differential operator $\caL^{0,1}$, we  first observe that $$\partial_{z} d(v,z)= c v\ d(v,z)\ .$$ Hence, the  self-duality relation for the single-edge generator $\caL^{1,0}_{x,y}$ rewrites
	\beq \nonumber
	-(v_x - v_y)(\const z_x - \const z_y) d(v_x, z_x) d(v_y, z_y)
	&=& - (z_x - z_y) (\const v_x -\const v_y) d(v_x, z_x) d(v_y, z_y)\ ,
	\eeq
	which indeed holds.
	
	For the second proof of self-duality for the single-site function
	\beq \nonumber
	d(v,z) &=& \pFq{0}{1}{-}{\alpha}{c v z}\ ,\quad v, z \in \R_+\ ,
	\eeq
	we use the following shortcut: for $x \in V$,
	$$
	F_x(\alpha) = \pFq{0}{1}{-}{\alpha}{c v_x z_x}\ .
	$$
	Additionally, we recall a formula for the $z_x$-derivative of $F_x$, namely
	\beq \label{derivative}
	\frac{\partial}{\partial z_x} F_x(\alpha) &=& \frac{c v_x}{\alpha} F_x(\alpha+1)\ ,
	\eeq
	and a recurrence identity
	\beq\label{recurrence}
	F_x(\alpha+1) &=& F_x(\alpha) - \frac{c v_x z_x}{\alpha (\alpha+1)} F_x(\alpha+2)\ .
	\eeq
	Hence, the self-duality relation  for $\caL^{1,\alpha}_{x,y}$ w.r.t.\ the function $F_x(\alpha) F_y(\alpha)$ rewrites by using \eqref{derivative} as
	\beq \nonumber
	&& c z_x v_y F_x(\alpha+1) F_y(\alpha) + c v_x z_y F_x(\alpha) F_y(\alpha+1) \\
	&+&\nonumber  \frac{c^2 (v_x z_x)(v_y z_x)}{\alpha (\alpha+1)} F_x(\alpha+2) F_y(\alpha)  + \frac{c^2 (v_x z_y)(v_y z_y)}{\alpha (\alpha+1)} F_x(\alpha)F_y(\alpha+2) \\ \nonumber
	&=&c v_x z_y F_x(\alpha+1) F_y(\alpha) + c v_y z_x F_x(\alpha) F_y(\alpha+1)\\
	&+&\nonumber   \frac{c^2 (v_x z_x) (v_x z_y)}{\alpha (\alpha+1)} F_x(\alpha+2) F_y(\alpha)  + \frac{c^2 (z_x v_y)(z_y v_y)}{\alpha (\alpha+1)} F_x(\alpha)F_y(\alpha+2)\ .
	\eeq
	By substituting \eqref{recurrence}, the identity holds.
	
	The proof for the operator $\caL^{-1,\gamma}$ follows the same lines and we omit it.
\end{proof}

\subsubsection*{Acknowledgments}
F.R.\ thanks Cristian Giardin\`{a} and Gioia Carinci for precious discussions; F.S.\ thanks Jan Swart for indicating  the point of view of an intertwining relation.  F.S.\ acknowledges NWO for financial support via the TOP1 grant 613.001.552.

\end{document}